\documentclass{svjour3}    
            
\smartqed  
\usepackage{amssymb,amsmath}
\usepackage{cases}
\usepackage{amsfonts,color}
\usepackage{graphicx}
\usepackage{epstopdf}
\usepackage{subfigure} 
\usepackage{algorithm}    
\usepackage{algpseudocode} 
\usepackage{booktabs}

\usepackage[misc,geometry]{ifsym} 
\usepackage[latin1]{inputenc}
\usepackage{subfigure}
\usepackage{tikz} 
\usetikzlibrary{matrix, graphs, arrows.meta, positioning, calc}
\usetikzlibrary{3d}
\usepackage[CJKbookmarks=true,
bookmarksnumbered,
bookmarks=true,
colorlinks,
linkcolor=red,
citecolor=cyan]{hyperref}
\hypersetup{colorlinks=true,citecolor=blue,linkcolor=red}%

\journalname{JOTA}

\begin{document}
	\title{
		Perturbation  analysis  on  T-eigenvalues of third-order tensors
	}
	
	\titlerunning{% if too long for running head
		Perturbation  analysis  on  T-eigenvalues of third-order tensors
	}
	
\author{Changxin Mo $^{1}$ \and Weiyang Ding $^{2}$ \and Yimin Wei$^{3}$ }
\authorrunning{C. Mo, W. Ding, Y. Wei } % if too long for running head
\institute{
		Communicated by Guoyin Li.\\
	\rule[1pt]{3.8cm}{0.5pt}	
		\begin{itemize}
			\item[\Letter] Yimin Wei
			\item[] \email{ymwei@fudan.edu.cn; yimin.wei@gmail.com}
			\item[] Changxin Mo
			\item[] \email{cxmo16@cqnu.edu.cn; cxmo16@fudan.edu.cn}
			\item[] Weiyang Ding 
			\item[] \email{dingwy@fudan.edu.cn}
			\item[1] School of Mathematical Sciences, Chongqing Normal University, Chongqing,  401331, P. R. of
			China. 
			\item[2] Institute of Science and Technology for Brain-Inspired Intelligence, Fudan University, Shanghai, China; Shanghai Center for Brain Science and Brain-Inspired Technology, Shanghai, China; Key Laboratory of Computational Neuroscience and Brain-Inspired Intelligence (Fudan University), Ministry of Education, China; MOE Frontiers Center for Brain Science, Fudan University, Shanghai, China; Zhangjiang Fudan International Innovation Center. 550025, P. R. of
			China. 
			\item[3] School of
			Mathematical Sciences and Shanghai Key Laboratory of Contemporary
			Applied Mathematics, Fudan University, Shanghai, 200433, P. R. of
			China.
		\end{itemize}
	}
	
	\date{Received: 26 July 2023 / Accepted: date}

\maketitle

\begin{abstract}
This paper concentrates on  perturbation theory concerning the tensor T-eigenvalues within the framework of tensor-tensor multiplication. Notably, it serves as a cornerstone for the extension of  semidefinite programming into the domain of tensor fields, referred to as T-semidefinite programming. 
The analytical perturbation analysis delves into the sensitivity of T-eigenvalues for third-order tensors with square frontal slices, marking the first main part of this study. Three classical results from the matrix domain into the tensor domain are extended. Firstly, this paper  presents the Gershgorin disc theorem for tensors, demonstrating the confinement of all T-eigenvalues within a union of Gershgorin discs.
Afterward, generalizations of the Bauer-Fike theorem are provided, each applicable to different cases involving tensors, including those that are F-diagonalizable and those that are not.
 Lastly, the Kahan theorem is presented, addressing the perturbation of a Hermite tensor by any tensors. Additionally, the analysis establishes connections between the T-eigenvalue problem and various optimization problems. 
The second main part of the paper focuses on tensor pseudospectra theory, presenting four equivalent definitions to characterize tensor $\varepsilon$-pseudospectra. Accompanied by a thorough analysis of their properties and illustrative visualizations, this section also explores the application of tensor $\varepsilon$-pseudospectra  in identifying  more T-positive definite tensors.
\end{abstract}

%\vspace*{0.25cm}

%%%%%%%%%%%%%%%%%%%%%%%%%%%%%%%%%%%%%%%%%%%%%%%%%%%%%%%%%%%%%%%%%%%%%%%%%%%%%%%%%%%%%%%%%%%%%%%%%%%%%%%%%%%%%%%%%%%%%%%%%%%%%%%%%%%%%%%%%%%%%%%%%%%%%%%%%%%%%%%%%%%%%%%%%%%%%%%%%%%%%%%%%%%%%%%
\section{Introduction}

Semidefinite programming  is a powerful optimization framework that generalizes linear and quadratic programming to handle a broader class of optimization problems. A recent noteworthy advancement involves the extension of semidefinite programming into the domain of tensor fields, denoted as T-semidefinite programming, a pivotal concept attributed to the  contributions of Zheng, Huang, and Wang \cite{zheng2020t}.  Minimizing the maximum T-eigenvalue of a third-order symmetric tensor exemplifies an important and  specific instance of the T-semidefinite programming:
\begin{equation}\label{TSDP}
	\max _{\eta, \mathbf{z}} \quad     -\eta \quad \text { s.t. } \quad \eta \mathcal{E}_{nnp}-\mathcal{A}(\mathbf{z}) \succeq_{\mathcal{T}} \mathcal{O} ,
\end{equation}
where the identity tensor $\mathcal{E}_{nnp} $ and zero tensor $\mathcal{O} $ are in $\mathbb{R}^{n\times n\times p}$, $\mathcal{A}(\mathbf{z}) \in \mathbb{R}^{n\times n\times p}$ is a symmetric tensor depends linearly on a vector $\mathbf{z}$, and $\mathcal{A} \succeq_{\mathcal{T}} \mathcal{O}$ means that $\mathcal{A}$ belongs to the set of symmetric T-positive semidefinite tensors. 
Notably,  T-semidefinite programming  is particularly pertinent in addressing the T-positive semidefiniteness problem of third-order symmetric tensors, which further involving with T-eigenvalues closely.  For example, the problem \eqref{TSDP} is equivalent with the problem that the minimum T-eigenvalue of tensor $\mathcal{E}_{nnp}-\mathcal{A}(\mathbf{z})$ is nonnegative, i.e., $\lambda_{\min }(\eta\mathcal{E}_{nnp}-\mathcal{A}(\mathbf{z})) \geq 0$. To sum up,  it is evident that T-eigenvalues exhibit an inherent connection with the T-positive semidefiniteness problem, thereby assuming a pivotal role in certain T-semidefinite programming problems within the field of optimization.

The approach of employing pseudospectra localizations for H- and Z-eigenvalues \cite{brazell_solving_2013,mo2019z,qi2005eigenvalues}, aimed at identifying more positive definite tensors, has been investigated by many researchers \cite{KostiPseudospectra2016,LiLiuWei2019}. However, to the best of the authors' knowledge, there has been no literature  systematically considering perturbation analysis related to T-eigenvalues of third-order tensors.
Perturbation theory, a field of study spanning over ninety years, originated from the works of Rayleigh \cite{Rayleigh1927} and Schr\"odinger  \cite{Schrodinger1926} as they investigated eigenvalue problems in vibrating systems and quantum mechanics, respectively. Research on this topic can be broadly categorized into two streams.  One stream focuses on analytic perturbation theory, while the other stream delves into matrices and addresses perturbation bounds. 
Extensive investigations and  pioneering works are encouraged to consult the recent paper \cite{greenbaum2020firstorder} and renowned monographs by Rellich, Sun, Stewart, Kato, and others \cite{kato2013,Rellich1969,stewartsun1990matrix,Sun1987}.

The multiplication of tensors, a fundamental and crucial operation analogous to matrix multiplication, has garnered considerable attention across various scientific disciplines.
In 2008, Kilmer \textit{et al.} \cite{Kilmer2008third} introduced a novel form of tensor multiplication that enables the representation of a third-order tensor as a product of other  third-order tensors. This development stemmed from their endeavor to extend the matrix singular value decomposition  to the realm of tensors. For the purpose of clarification and distinction from other tensor product operations \cite{Golub2013matrix}, this specific type of multiplication is referred to as {\it tensor-tensor multiplication}.
This new tensor-tensor multiplication involves three closely interconnected operators. Let $\mathbb{C}$ denote the complex field. For a third-order tensor $\mathcal{A} \in \mathbb{C}^{n_{1} \times n_{2} \times n_{3}}$, the frontal slices, denoted as $A_{::k}$ or simply $A_k$, can be obtained by fixing the last index. Consequently, the tensor has $n_3$ frontal slices, denoted as $A_1, \ldots, A_{n_3}$, each of which represents a matrix of dimensions $n_{1} \times n_{2}$.   The first operation \text{bcirc}($\cdot$) we will introduce involves the creation of a block circulant matrix using all the frontal slices of a tensor.
That is,
\begin{equation}\label{Definition bcirc}
	\operatorname{bcirc}(\mathcal{A})
	=
	\left[\begin{array}{ccccc}
		{A_{1}} & {A_{n_{3}}} & {A_{n_{3}-1}} & {\cdots} & {A_{2}} \\
		{A_{2}} & {A_{1}} & {A_{n_{3}}} & {\cdots} & {A_{3}} \\
		{A_{3}} & {A_{2}} & {A_{1}} & {\ddots} & {\vdots} \\
		{\vdots} & {\vdots} & {\ddots} & {\ddots} &   {A_{n_{3}}} \\
		{A_{n_{3}}} & {A_{n_{3}-1}} &  {\cdots} &{A_{2}} & {A_{1}}
	\end{array}\right]
\end{equation} 
with size $(n_1n_3) \times (n_2n_3)$.  The tensor representation described here is orientation-specific, which proves useful for applications where the data possesses a fixed orientation, such as time series applications \cite{Kilmer2011}.
The remaining two operations can be considered as the ``inverse" of each other. They are the \text{unfold}($\cdot$) and \text{fold}($\cdot$) commands, defined as follows,
\begin{equation*}
	\operatorname{unfold}(\mathcal{A})=\left[\begin{array}{c}
		A_{1} \\
		A_{2} \\
		\vdots \\
		A_{n_{3}}
	\end{array}\right], \quad \operatorname{fold}(\operatorname{unfold}(\mathcal{A}))=\mathcal{A}.
\end{equation*}
It is observed that the $\operatorname{unfold}(\cdot)$ command transforms a tensor into a matrix, while the $\operatorname{fold}(\cdot)$ command reverses this process by converting the matrix back into a tensor. 
On the basis of the  three operations, \text{bcirc}($\cdot$), \text{unfold}($\cdot$), and \text{fold}($\cdot$), the tensor-tensor multiplication of two tensors, $\mathcal{B} \in \mathbb{C}^{n_{1} \times p \times n_{3}}$ and $\mathcal{C} \in \mathbb{C}^{p \times n_{2} \times n_{3}}$, is defined as follows:
\begin{equation}
	\mathcal{B} * \mathcal{C} :=\operatorname{fold}(\mathrm{bcirc}(\mathcal{B}) \operatorname{unfold}(\mathcal{C})). \label{tensor-tensor multiplication}
\end{equation}
Let $\mathcal{A} = \mathcal{B} * \mathcal{C}$. It can be easily verified that $\mathcal{A} \in \mathbb{C}^{n_{1} \times n_2 \times n_{3}}$.  On the size of the two tensors involved in \eqref{tensor-tensor multiplication} and also their product $\mathcal{A}$, special attention should be given to the first two indices of $\mathcal{B}$ and $\mathcal{C}$ since their frontal slices need to be consistent with matrix multiplication.
The usefulness of tensor-tensor multiplication \eqref{tensor-tensor multiplication} has been demonstrated in various domains in recent years, including but not limited to image processing (such as image deblurring and compression, object and facial recognition), tensor principal component analysis, tensor completion, multilinear control systems, and pattern recognition; see  \cite{cao2022perturbation,chen2023perturbations,cui2021perturbation,han2023t,kilmer2013third,kilmertensortensor,liu2018improved,Ng2022,mo2020time,Kilmer2020ima,tang2023sketch,wang2023solving,wang2022hot,wu2020graph,ZHAO2020137}
and the references therein. Various topics, such as linear complementarity problems and tensor complementarity problems based on tensor-tensor multiplication, have also been studied in recent years \cite{wang2023fixed,wei2023neural}.

The generalization of eigenvalues from matrices to tensors has been studied  through the implementation of  tensor-tensor multiplication. Significant attention and extensive research have been devoted to this field, resulting in a substantial body of work focused on their variants, applications, and theoretical analysis. 
In 2011,  Kilmer, Braman, and Hao \cite{Kilmer2011report} studied various decompositions, such as eigendecomposition, tensor decomposition, and so on.  In \cite{Lund2020}, Lund defined a tensor eigendecomposition for third-order tensors with diagonalizable faces. 
Then, the notion of T-eigenvalues was introduced by Miao, Qi and Wei \cite{Miao2020T} and also Liu and Jin \cite{Jin2020}, establishing a fundamental and significant concept.  Alternative versions and formulations of  eigenvalues of third-order tensors in the context of tensor-tensor multiplication have also been explored by Qi and Zhang \cite{qi2021t}, who referred to them as``eigentuples", and by Beik and Saad \cite{saad2023}, who termed them as ``tubular eigenvalues".   A comprehensive investigation on the relationships between tubular eigenvalues, T-eigenvalues, and eigentuples has been conducted by Beik and Saad \cite{saad2023}. 
 The T-eigenvalues also exhibit a multitude of applications across diverse mathematical domains.  The T-eigenvalues were also utilized by Zheng \textit{et al.} \cite{zheng2020t} to study  the T-positive semidefiniteness  and T-semidefinite programming problems. They also show that T-eigenvalues have a close relationship with many optimization problems. 
The stability of T-eigenvalues was addressed in \cite{Jin2020} to analyze the tensor Lyapunov equation commonly encountered in spatially invariant systems. 
Furthermore, several results from the matrix domain have been extended to the tensor domain. Notable results include the Weyl's and Cauchy's interlacing theorems, the arithmetic-geometric mean inequality, H{\"o}lder inequality, and Minkowski inequality \cite{cao2021tensor,Jin2020}, the inequalities and probability bounds  \cite{chang2022t2,chang2022t}, the Perron-Frobenius type theorem  \cite{yang2023perron}, and the numerical range \cite{pakmanesh2022m}.
 The locus of singular tuples of a complex-valued multi-symmetric
tensor has also been studied in  \cite{Turatti2022}.   Recently, the authors in \cite{el2023spectral} have also conducted a study on spectral computation.
Also, numerous researchers have investigated the properties and functions of multidimensional arrays within the framework established by tensor-tensor multiplication \cite{liu2022weighted,Lund2020,lund2023frechet,Miao2020generalized,Miao2020T,miao2022stochastic}.

The study of T-eigenvalues has emerged as a prominent research area within the field of tensor analysis. Motivated by the aforementioned research, we pay our attention to the perturbation analysis of third-order tensors under the novel tensor-tensor multiplication \eqref{tensor-tensor multiplication} in this paper, encompassing both the extension of classical theoretical results and the introduction of novel pseudospectra theory.
Due to many scholars have focused their attentions on the matrix perturbation analysis \cite{Bauer-Fike,1986Generalization,Rellich1969,shi2012sharp,Sun1987,trefethen2005spectra}, a wealth of results have been developed up to now. These include the Gershgorin disc theorem, the Bauer-Fike theorem,  and the Kahan theorem, as well as the development of pseudospectra theory for matrices. 
In this paper, we generalize these classical results and pseudospectra theory to the tensor case. Specifically, we present a Gershgorin disc-type theorem for tensors of size $m \times m \times n$, demonstrating that all T-eigenvalues lie within a union of Gershgorin discs (cf. Theorem \ref{GerThm}). Compared to similar results in \cite{cao2021tensor}, we obtain tighter bounds under certain conditions.
Moreover, we provide three generalizations of the Bauer-Fike theorem to the tensor case. The first generalization extends the classical Bauer-Fike theorem for matrices \cite{Bauer-Fike,Golub2013matrix} to F-diagonalizable tensors of size $m \times m \times n$ under different norms (cf. Theorem \ref{Bauer Ten}). Two additional generalizations are developed for $m \times m \times n$ tensors, which may not be F-diagonalizable (cf. Theorems \ref{Theorem Schur} and \ref{General two}). These can be viewed as generalizations of the Bauer-Fike theorems for non-diagonalizable matrices presented in \cite{1986Generalization,Golub2013matrix}.
Furthermore, we study the generalization of the Kahan theorem to the tensor case, considering general perturbations on Hermite tensors (cf. Theorem \ref{KahanThm}).
The second main contribution of this paper is the development of pseudospectra theory for third-order tensors. We present four different definitions of tensor $\varepsilon$-pseudospectra (cf. Definitions \ref{Definition pseudospectra} and \ref{Definition pseudospectra-singularValue}) and establish their equivalence under certain conditions. We also provide various pseudospectral properties.
Finally, we present visualizations and the application of the $\varepsilon$-pseudospectra of certain tensors through numerical examples.

The remaining sections of this paper are organized as follows. In Section \ref{Preliminaries}, we introduce some notations commonly used throughout the paper and review  basic concepts and fundamental results.
Section \ref{Perturbationanalysis section} is dedicated to the perturbation analysis of third-order tensors and represents one of the main parts of this study. We extend several classical theorems from the matrix domain to the tensor domain, offering insights into the perturbation behavior of tensors.
Section \ref{Pseudospectra sec} delves into the topic of $\varepsilon$-pseudospectral theory for tensors, which forms the second main part of this paper. We investigate various aspects of pseudospectra theory, exploring different definitions of tensor $\varepsilon$-pseudospectra and discussing their properties. Additionally, we complement our analysis by presenting visualizations that depict the $\varepsilon$-pseudospectra via some numerical examples.
Finally, in the concluding section, we summarize the key findings and contributions of this paper.

%%%%%%%%%%%%%%%%%%%%%%%%%%%%%%%%%%%%%%%%%%%%%%%%%%%%%%%%%%%%%%%%%%%%%%%%%%%%%%%%%%%%%%%%%%%%%%%%%%%%%%%%%%%%%%%%%%%%%%%%%%%%%%%%%%%%%%%%%%%%%%%%%%%%%%%%%%%%%%%%%%%%%%%%%%%%%%%%%%%%%%%%%%%%%%%
\section{Preliminaries}\label{Preliminaries}

We begin by presenting the notations that will be utilized throughout the paper. Subsequently, we provide a comprehensive review of fundamental tensor concepts, including the identity tensor, tensor transpose, F-diagonal tensor, orthogonal tensor, and others. These concepts are defined within the framework established by Kilmer \textit{et al.} \cite{kilmer2013third,Kilmer2011,Kilmer2008third}, which centers around the tensor-tensor multiplication operation.

In general, scalars are represented by lowercase letters, e.g., $a$. Vectors and matrices are denoted by boldface lowercase letters and capital letters, respectively, e.g., $\mathbf{v}$ and $A$. Euler script letters are utilized to represent higher-order tensors, such as $\mathcal{A}$. The frontal slices of a tensor $\mathcal{T} \in \mathbb{C}^{n_1\times n_2 \times n_3}$ are denoted by corresponding capital letters with subscripts, i.e., $T_1, \ldots, T_{n_3}$, which are matrices of size $n_1\times n_2$. For a matrix $A$, $A^{\top}$ (or $A^{\mathrm{H}}$) denotes its transpose (or conjugate transpose). 
The Kronecker product, denoted by `` $\otimes$'', is an operation performed on two matrices. Let $A=(a_{ij}) \in \mathbb{C}^{m\times n}$ and $B\in \mathbb{C}^{p\times q}$ be the two matrices. The Kronecker product yields a block matrix of size $(pm) \times (qn)$. That is, 
$$
A \otimes B=\left[\begin{array}{ccc}
	a_{11} B & \cdots & a_{1 n} B \\
	\vdots & \ddots & \vdots \\
	a_{m 1} B & \cdots & a_{m n} B
\end{array}\right].
$$
Throughout the paper, we frequently employ the $n\times n$ identity matrix, denoted as  $I_n$, as well as the $n\times n$ normalized discrete Fourier transform (DFT) matrix, denoted as $F_n$, for convenience.
The normalized  DFT matrix $F_n$  is defined as:
\[
F_n = \frac{1}{\sqrt{n}} \begin{bmatrix}
	1 & 1 & 1 & \ldots & 1 \\
	1 & \omega & \omega^2 & \ldots & \omega^{n-1} \\
	1 & \omega^2 & \omega^4 & \ldots & \omega^{2(n-1)} \\
	\vdots & \vdots & \vdots & \ddots & \vdots \\
	1 & \omega^{n-1} & \omega^{2(n-1)} & \ldots & \omega^{(n-1)(n-1)}
\end{bmatrix}
\]
where  $\omega = e^{-\frac{2\pi i}{n}}$ is the complex $n$th root of unity. The normalization factor $\frac{1}{\sqrt{n}}$ ensures that the DFT matrix is unitary.

Next, we revisit some fundamental concepts of tensors that are frequently used in the subsequent sections of this paper.

\begin{definition}{\rm (\cite[Definition 3.14]{Kilmer2011}, transposed tensor).}
	Let $\mathcal{A} \in \mathbb{C}^{n_1\times n_2 \times n_3}$.  Then the transposed tensor  $\mathcal{A}^{\top}  \in \mathbb{C}^{n_2\times n_1 \times n_3}$ or conjugate transposed tensor  $\mathcal{A}^{\mathrm{H}}  \in \mathbb{C}^{n_2\times n_1 \times n_3}$  is obtained by taking the  transpose or conjugate transpose  of each of the frontal slices and then reversing the order of transposed frontal slices $2$ through $n_3$. 
\end{definition}

Based on the aforementioned definition, a tensor $\mathcal{A}\in \mathbb{C}^{m\times m\times n}$ is said to be symmetric when $\mathcal{A} = \mathcal{A}^{\top}$, or Hermitian when $\mathcal{A} = \mathcal{A}^{\mathrm{H}}$ \cite{Jin2020}.

\begin{definition}{\rm (\cite[Definition 3.4]{Kilmer2011}, {\rm identity tensor}).}
	Let $\mathcal{E}_{mm\ell} \in \mathbb{C}^{m\times m \times \ell}$. If its frontal slice $E_1$ is the  identity
	matrix of size $m \times m$, and whose other frontal slices $E_2, \ldots, E_{\ell}$ are all zeros, then we call  $\mathcal{E}_{mm\ell}$ an identity tensor.
\end{definition}

\begin{definition}{\rm (\cite[Definition 3.5]{Kilmer2011}, {\rm inverse of a tensor}).}
	%{\rm (Inverse of a tensor)}
	Let $\mathcal{A} \in \mathbb{C}^{m\times m \times \ell}$. The tensor $\mathcal{B} \in \mathbb{C}^{m\times m \times \ell}$ is referred to as the inverse of $\mathcal{A}$ if it fulfills the following two equations,
	$$
	\mathcal{A} * \mathcal{B}=\mathcal{E}_{mm\ell}, \quad \text { and } \quad \mathcal{B} * \mathcal{A}=\mathcal{E}_{mm\ell}.
	$$
	When such conditions hold, the tensor $\mathcal{A}$ is considered invertible, and its inverse is denoted as $\mathcal{A}^{-1}$.
\end{definition}

\begin{definition}{\rm (\cite[Definition 3.18]{Kilmer2011}, 	{\rm orthogonal  tensor and unitary tensor}).} 
	Let $\mathcal{Q} \in \mathbb{R}^{m\times m \times \ell}$. We call   $\mathcal{Q}$ an orthogonal tensor provided that
	$
	\mathcal{Q}^{\top} * \mathcal{Q}= \mathcal{Q} * \mathcal{Q}^{\top}  =   \mathcal{E}_{mm\ell}. 
	$
	If $\mathcal{Q} \in \mathbb{C}^{m\times m \times \ell}$ and $
	\mathcal{Q}^{\mathrm{H}}* \mathcal{Q}= \mathcal{Q} * \mathcal{Q}^{\mathrm{H}}  =   \mathcal{E}_{mm\ell}$, then we call it a unitary tensor.
\end{definition}

We call a third-order tensor $\mathcal{D}\in \mathbb{C}^{m\times m \times \ell}$ an F-diagonal tensor if all its frontal slices $D_1, \ldots, D_{\ell}$  are  diagonal matrices \cite{Kilmer2011}. 

\begin{definition}{\rm (\cite{kilmer2013third,Lund2020,Miao2020T}, F-diagonalizable tensor). }
	Let  $\mathcal{A}\in \mathbb{C}^{m\times m \times \ell}$. If there exists an invertible tensor $\mathcal{P}$ and an F-diagonal tensor $\mathcal{D}$ such that 
	$$
	\mathcal{A}=\mathcal{P} * \mathcal{D} * \mathcal{P}^{-1},
	$$
	then we refer to $\mathcal{A}$ as an F-diagonalizable tensor.
\end{definition}

\begin{definition}{\rm (\cite[Definition 6]{zheng2020t}, T-positive (semi)definite tensor). }
	Let  $\mathcal{A}\in \mathbb{R}^{m\times m \times \ell}$. A tensor $\mathcal{A}$ is called a symmetric T-positive (semi)definite 	tensor if and only if it is  symmetric and 
	$$
	\langle\mathcal{X}, \mathcal{A} * \mathcal{X}\rangle>(\geq) 0,
	$$
	where $\langle\mathcal{A}, \mathcal{B}\rangle:=\sum_{i, j, k} a_{i j k} b_{i j k}$,
holds	for any nonzero $\mathcal{X} \in \mathbb{R}^{m\times 1 \times \ell}$ (for any  $\mathcal{X} \in \mathbb{R}^{m\times 1 \times \ell}$). 
\end{definition}

It has been shown that a symmetric third-order tensor $\mathcal{A}$ is T-positive (semi)definite  if and only if each T-eigenvalue of $\mathcal{A}$ is positive (nonnegative) \cite{zheng2020t}.

Several useful lemmas are summarized as follows.

\begin{lemma}\label{Lemmabcirc} {\rm (\cite{Jin2020,Lund2020,Miao2020generalized}). }
	The following results hold for third-order tensors $\mathcal{A} \in \mathbb{C}^{m \times n \times p}$:
	
	{\rm (a)} The operator  \text{bcirc}($\cdot$) defined in  (\ref{Definition bcirc})
	is a linear operator, i.e., 
\begin{equation*}
	\operatorname{bcirc}(\alpha \mathcal{A}+\beta \mathcal{B})=\alpha \operatorname{bcirc}(\mathcal{A})+\beta \operatorname{bcirc}(\mathcal{B}),
\end{equation*}
	where $\mathcal{B}$ has the same size as $\mathcal{A}$  and $\alpha, \beta$ are constants.
	
	{\rm (b)}   $\operatorname{bcirc}(\mathcal{A} * \mathcal{C})=\operatorname{bcirc}(\mathcal{A}) \operatorname{bcirc}(\mathcal{C})$  where $\mathcal{C} \in \mathbb{C}^{n \times s \times p}$. 
	
	{\rm (c)}  $\operatorname{bcirc}\left(\mathcal{A}^{\top}\right)=(\operatorname{bcirc}(\mathcal{A}))^{\top}$,  and $\operatorname{bcirc}\left(\mathcal{A}^{\mathrm{H}}\right)=(\operatorname{bcirc}(\mathcal{A}))^{\mathrm{H}}$.
	
	{\rm (d)}   The inverse  of an invertible $\mathcal{A}$ is unique and $
	\mathrm{bcirc}\left(\mathcal{A}^{-1}\right)=( \operatorname{bcirc}(\mathcal{A}))^{-1}.
	$
\end{lemma}

\begin{lemma}{\rm (\cite[Theorems 2.7 and 2.8]{Jin2020}).} \label{Properties of symmetric tensor}
	Let  $\mathcal{A}\in \mathbb{C}^{m\times m\times n}$.  Some fundamental results involving Hermitian or symmetric tensors are:
	
	{\rm (a)} The tensor $\mathcal{A}$ is symmetric if and only if $\operatorname{bcirc}\left(\mathcal{A}\right)=(\operatorname{bcirc}(\mathcal{A}))^{\top}$. 
	
	{\rm (b)}  The tensor $\mathcal{A}$ is Hermitian if and only if $\operatorname{bcirc}\left(\mathcal{A}\right)=(\operatorname{bcirc}(\mathcal{A}))^{\mathrm{H}}$. 
	
	{\rm (c)} All T-eigenvalues (cf. Definition \ref{def3-1}) of a Hermitian tensor $\mathcal{A}$ are real.
\end{lemma}

\begin{lemma}{\rm (\cite[Lemma 4]{Miao2020T}).}\label{lemmaFdiagonal iff}
	Suppose $A_{1}, \cdots, A_{p}, B_{1},  \cdots, B_{p} \in \mathbb{C}^{n \times n}$ are  matrices satisfying
	$$
	\left[\begin{array}{cccc}
		A_1 & A_p &  \cdots & A_{2} \\
		A_{2} & A_{1}  & \cdots & A_{3} \\ \vdots & \vdots &\ddots & \vdots \\ 
		A_{p} & A_{p-1} &  \cdots  & A_{1}
	\end{array}\right]
	=
	\left(F_{p} \otimes I_{n}\right)
	\left[\begin{array}{cccc}
		B_{1} & & & \\ 
		& B_{2} & & \\ 
		& & \ddots & \\ 
		& & & B_{p}
	\end{array}\right]\left(F_{p}^{\mathrm{H}} \otimes I_{n}\right).
	$$
	Then $B_{1},  \cdots, B_{p}$ are diagonal (sub-diagonal, upper-triangular, lower-triangular) matrices if and only if $A_{1},  \cdots, A_{p}$ are diagonal (sub-diagonal, upper-triangular, lower-triangular) matrices.
\end{lemma}

A tensor $\mathcal{A}\in \mathbb{C}^{m\times m \times n}$ is  normal if it satisfies the condition $\mathcal{A}*\mathcal{A}^{\mathrm{H}} = \mathcal{A}^{\mathrm{H}} * \mathcal{A}$, indicating that $\mathcal{A}$ commutes with its conjugate transpose under tensor-tensor multiplication, as stated in \cite{Miao2020T}. It is evident that symmetric and Hermitian tensors are examples of normal tensors.

As stated in the subsequent conclusion, any normal third-order tensor can be F-diagonalizable by means of a unitary tensor.
\begin{lemma}{\rm (\cite{Miao2020T}).}\label{normal diagonal}
	For a given normal tensor $\mathcal{A}\in \mathbb{C}^{m\times m \times n}$, there exists a unitary tensor $\mathcal{U}\in \mathbb{C}^{m\times m \times n}$ such that
	$$
	\mathcal{A} = \mathcal{U} * \mathcal{D} * \mathcal{U}^{\mathrm{H}},
	$$
	where $\mathcal{D}$ is an F-diagonal tensor.
	
\end{lemma}

%%%%%%%%%%%%%%%%%%%%%%%%%%%%%%%%%%%%%%%%%%%%%%%%%%%%%%%%%%%%%%%%%%%%%%%%%%%%%%%%%%%%%%%%%%%%%%%%%%%%%%%%%%%%%%%%%%%%%%%%%%%%%%%%%%%%%%%%%%%%%%%%%%%%%%%%%%%%%%%%%%%%%%%%%%%%%%%%%%%%%%%%%%%%%%%
\section{Perturbation Analysis on Third-Order Tensors} \label{Perturbationanalysis section}

Within the framework of tensor-tensor multiplication (\ref{tensor-tensor multiplication}) proposed and investigated by Kilmer and Martin \cite{Kilmer2011}, T-eigenvalues and T-eigenvectors have garnered significant attention from researchers. They offer a novel perspective to characterize the properties of the widely employed tensor-tensor multiplication (\ref{tensor-tensor multiplication}). Extensive studies from diverse viewpoints can be found in the references \cite{braman2010thirdorder,Kilmer2013SIAM,Jin2020,Miao2020T,zheng2020t}. Before proceeding with our main results, it is essential to revisit the fundamental concept of T-eigenvalues and T-eigenvectors. Subsequently, we further elaborate on this concept by extending it to a more general case, known as generalized T-eigenvalues and generalized T-eigenvectors. At last,  the connection between T-eigenvalues and certain optimization problems is presented.

\begin{definition}{\rm (\cite[Definition 2.5]{Jin2020}).}\label{def3-1}
	Suppose that $\mathcal{A}$ is a tensor with size $m \times m \times \ell$. If there exists a nonzero $m \times 1 \times \ell$ tensor $\mathcal{X}$ and scalar $\lambda$ such that
	\begin{equation}\label{def-T-eig}
		\mathcal{A} * \mathcal{X}  = \lambda  \mathcal{X}, 
	\end{equation}
	then $\lambda$ is called a T-eigenvalue of the third-order tensor $\mathcal{A}$ and $\mathcal{X}$ is a T-eigenvector of $\mathcal{A}$ associated to $\lambda$.
\end{definition}

\begin{remark}
	We can further extend the concept of T-eigenvalues into  generalized T-eigenvalues, similar to the case of generalized matrix eigenvalues. Let $\mathcal{B}$ be another tensor with the same size as $\mathcal{A}$. Under the same conditions as defined in Definition \ref{def3-1}, if the following equation holds:
	\begin{equation}\label{Gene eig}
		\mathcal{A} * \mathcal{X} = \lambda (\mathcal{B} * \mathcal{X}) ,
	\end{equation}
	then $\lambda$ is referred to as a generalized T-eigenvalue of $\mathcal{A}$ with respect to $\mathcal{B}$. This generalization encompasses the cases presented in prior works \cite{braman2010thirdorder,Kilmer2013SIAM,Jin2020,Miao2020T,zheng2020t}.
	
	Preliminary investigations into generalized T-eigenvalues \eqref{Gene eig} are conducted in the following analysis. Utilizing the definition of tensor-tensor multiplication as provided in (\ref{tensor-tensor multiplication}) and note that $ ( \lambda\mathcal{B} * \mathcal{X})= \lambda (\mathcal{B} * \mathcal{X})$ as well as $\lambda  \operatorname{bcirc}(\mathcal{B}) =  \operatorname{bcirc}( \lambda\mathcal{B})$, it can be observed that the equation (\ref{Gene eig}) is equivalent to
	\begin{equation}\label{Gen to matrix}
		\operatorname{bcirc}(\mathcal{A}) \operatorname{unfold}(\mathcal{X})=\lambda ( \operatorname{bcirc}(\mathcal{B}) \operatorname{unfold}(\mathcal{X})).
	\end{equation}
	It is important to note that $\operatorname{unfold}(\mathcal{X})$ represents a vector with dimensions $m\ell$. Consequently, the generalized T-eigenvalue problem (\ref{Gene eig}), based on tensor-tensor multiplication, exhibits a strong correlation with the classical generalized matrix eigenvalue problem \cite{Horn2013}. Accordingly, there exist $m\ell$ eigenvalues, accounting for multiplicities, for the problem (\ref{Gene eig}) if and only if
	$ \operatorname{rank}(\operatorname{bcirc}(\mathcal{B})) = m\ell.
	$
	In cases where the block circulant matrix $\operatorname{bcirc}(\mathcal{B})$ is deficient in rank, the set of generalized T-eigenvalues of $\mathcal{A}$ relative to $\mathcal{B}$ may be finite, empty, or infinite.
\end{remark}

The T-eigenvalue problem can be formulated based on the well-known eigenvalue problems associated with block circulant matrices, which finds wide applications \cite{Davis1979circulant}, and has a close relationship with optimization problems.  For example, the paper  \cite{Olson2014circulant} offers a concise tutorial on circulant matrices, exploring their application in modeling and analyzing the free and forced vibration of mechanically cyclic symmetrical structures.
Several example applications are provided, including the  eigenvalue problems with the system matrices being block circulant with square blocks.  Suppose $C$ is a $(m\ell) \times (m\ell)$ block circulat matrix with generating matrices $\{C_1, C_2, \cdots, C_\ell\}$ in which $C_i$ is of size $m\times m$ for $i=1,2,\cdots, \ell$, that is,  
\begin{equation*}
	C=\left[\begin{array}{cccc}
		C_1 & C_\ell & \cdots & C_2 \\
		C_2 & C_1 & \cdots & C_{3} \\
		\vdots & \vdots & \ddots & \vdots \\
		C_\ell & C_{\ell-1} & \cdots & C_1
	\end{array}\right],
\end{equation*}
and we consider the  following optimization problem, under the condition $C$ is symmetric (which yields to a symmetric tensor by reformulating),  with the variable $\mathbf{y} \in \mathbb{R}^{m\ell}$: 
\begin{equation} \label{opti-t1}
	\max _{\mathbf{y}} \quad     \mathbf{y}^{\top} C \mathbf{y}
	 \quad \text { s.t. } \quad 
	\mathbf{y}^{\top} \mathbf{y}=1 ,
\end{equation}
or the problem related to the maximum of the Rayleigh quotient, 
\begin{equation} \label{opti-t2}
	\max _{\mathbf{y}} \quad   \frac{\mathbf{y}^{\top} C \mathbf{y}}{\mathbf{y}^{\top}  \mathbf{y}}.  
\end{equation}
By using the method of Lagrange multipliers, we get the Lagrangian with Lagrange multiplier $\lambda \in \mathbb{R}$,
$
\mathfrak{L} =  \mathbf{y}^{\top} C \mathbf{y} - \lambda (\mathbf{y}^{\top} \mathbf{y}-1),
$
and letting the derivative of Lagrangian to zero give us $\mathbb{R}^{ml}\ni \frac{\partial \mathfrak{L}}{\partial \mathbf{y}} = 2C\mathbf{y} - 2\lambda \mathbf{y} = 0$, and thus
 $C\mathbf{y} = \lambda \mathbf{y}$. Note that $\mathbf{y} = \{y_1, \cdots, y_m, y_{m+1},\cdots, y_{2m}, \cdots, y_{m\ell}\}^{\top}$ and $C$ has a block circulant  structure, by resizing the vector $\mathbf{y}$ into a matrix 
\begin{equation*}
Y=\left[\begin{array}{cccc}
		y_1 & y_{m+1} & \cdots & y_{(\ell-1)m+1} \\
		y_2 & y_{m+2} & \cdots &  y_{(\ell-1)m+2} \\
		\vdots & \vdots & \ddots & \vdots \\
		y_m & y_{2m} & \cdots &  y_{\ell m}
	\end{array}\right] \in \mathbb{R}^{m\times \ell},
\end{equation*}
which can also be considered as a tensor $\mathcal{X}$ with size $m\times 1 \times \ell$, and also generating a tensor $\mathcal{A}$, which would be symmetric under the condition $C$ is symmetric, with the generating matrices $\{C_1, C_2, \cdots, C_\ell\}$ as its frontal slices in order, then we can get an equivalent T-eigenvalue form as \eqref{def-T-eig}, which illustrates that  the optimization problems \eqref{opti-t1} and \eqref{opti-t2} involving with block circulant  system matrices  employed in  vibration analysis \cite{Davis1979circulant,Olson2014circulant}  can  yield to the T-eigenvalue problem with symmetric tensor $\mathcal{A}$ defined in Definition \ref{def-T-eig}.

This section next  focuses on the perturbation analysis of T-eigenvalues, which represents a less complex scenario compared to \eqref{Gene eig} due to the assurance that the set of T-eigenvalues is neither empty nor infinite. By considering the identity tensor $\mathcal{E}$ as $\mathcal{B}$ in \eqref{Gen to matrix}, we can derive an equivalent form of \eqref{def-T-eig} as follows:
$
\operatorname{bcirc}(\mathcal{A}) \operatorname{unfold}(\mathcal{X})=\lambda \operatorname{unfold}(\mathcal{X}).
$
This implies that all T-eigenvalues of tensor $\mathcal{A}$ are, in fact, eigenvalues of the block circulant matrix $\operatorname{bcirc}(\mathcal{A})$, and vice versa. Notably, in the study of matrix perturbation theory, the Gershgorin disc, Bauer-Fike, and Kahan theorems \cite{Bauer-Fike,1986Generalization,Golub2013matrix,shi2012sharp,stewartsun1990matrix,Sun1987} hold significant importance as fundamental results in analyzing the sensitivity of eigenvalues of matrices to perturbations. Drawing on the foundations laid above, we extend these theorems to the tensor case.

\subsection{Gershgorin Disc Theorem for Tensors}

The Gershgorin disc theorem constitutes a fundamental outcome for bounding the spectra of square matrices \cite{Golub2013matrix,Horn2013}. Within this subsection, we aim to broaden the applicability of this theorem by extending it to encompass the domain of third-order tensors.

Consider the tensor $\mathcal{A}$ belonging to the complex space $\mathbb{C}^{m \times m \times \ell}$.  Through the utilization of the normalized DFT matrix, the matrix $\operatorname{bcirc}(\mathcal{A})$ can be block-diagonalized \cite{Kilmer2011}.   However, it is important to note that the resulting block-diagonal matrix may not necessarily be diagonalizable. As an illustrative example,  the matrix $\operatorname{bcirc}(\mathcal{A})$ constructed from tensor $\mathcal{A}$ with the following three frontal slices is not diagonalizable:
$$
A_1 = \left(
\begin{array}{ccc}
	1 & 1 & 0 \\
	0 & 1 & 0 \\
	0 & 0 & 1 \\
\end{array}
\right), A_2 =\left(
\begin{array}{ccc}
	1 & 1 & 1 \\
	0 & 1 & 0 \\
	0 & 0 & 1 \\
\end{array}
\right), A_3=
\left(
\begin{array}{ccc}
	1 & 0 & 1 \\
	0 & 1 & 0 \\
	0 & 0 & 1 \\
\end{array}
\right).
$$
In general, by employing a series of appropriate similarity transformations, the block-diagonal matrix can progressively be approximated towards a diagonal matrix, indicating a tendency towards increased diagonal structure. At this point in the analysis,  an important question arises: to what extent do the diagonal elements of the final matrix approximate the T-eigenvalues of the original tensor $\mathcal{A}$?

We now present the Gershgorin disc theorem tailored specifically for third-order tensors.

\begin{theorem} {\rm (Gershgorin disc theorem for tensors).}\label{GerThm}
	Let  $\mathcal{A}\in \mathbb{C}^{m \times m \times n}$. Assume 
	\begin{equation*}
		X^{-1} (F_n^{\mathrm{H}} \otimes I_m) \operatorname{bcirc}(\mathcal{A}) 	(F_n \otimes I_m) X=D+F,
	\end{equation*}
	where $X$ is a similarity transformation matrix, $D = \operatorname{diag}(d_1, \ldots, d_{mn})$ and F has zero diagonal entries. Then we have 
	$$
	\Lambda({\mathcal{A}}) \subseteq \Gamma_{\mathcal{A}}:=\bigcup_{i=1}^{mn} \Theta_i,
	$$
	where 
	$\Lambda({\mathcal{A}})$ denotes the  set of all T-eigenvalues of tensor $\mathcal{A}$ and 
	\begin{equation}\label{GerDiscs}
		\Theta_i=\left\{z \in \mathbb{C}:\left|z-d_{i}\right| \leq \sum_{j=1}^{mn}\left|f_{i j}\right|\right\}, i = 1, \ldots, mn.
	\end{equation}
\end{theorem}   
\begin{proof}
	Suppose  $\lambda \in 	\Lambda({\mathcal{A}})$.  If $\lambda = d_i$ for some $i$, then  the conclusion is evidently valid.  It is important to acknowledge that the diagonal elements of $D$ may not necessarily correspond to the T-eigenvalues of $\mathcal{A}$. Hence, we make the assumption that
	$\lambda  \neq d_i$ for $i = 1, \ldots, mn$. 
	Notice that T-eigenvalues remain the same under both unitary and similarity transformations. Therefore, the matrix $I- (\lambda I - D)^{-1}F$ is singular, implying that $1$ is an eigenvalue of the matrix $(D-\lambda I)^{-1} F$.
On one hand,  it can be readily verified that
	\begin{equation*}
		\left\|(D-\lambda I)^{-1} F\right\|_{\infty} \geq 1,
	\end{equation*}
	or else 	the matrix $I- (\lambda I - D)^{-1}F$ will be  nonsingular.  On the other hand, we have
	\begin{equation*}
		\left\|(D-\lambda I)^{-1} F\right\|_{\infty} =
		\frac{1}{\left|d_{k}-\lambda\right|}
		\sum_{j=1}^{mn} \left|f_{k j}\right|
	\end{equation*}
	for some $k \in \{1,2,\ldots, mn\}$. Combining the above two inequalities, we get
	$\left|\lambda- d_{k}\right| \leq  \sum_{j=1}^{mn} \left|f_{k j}\right|,
	$
	which further implies $\lambda \in 	\Theta_k$.  Note that $\lambda$ is arbitrary and we complete the proof.\qed
	
\end{proof}

\begin{figure}[h] 
	\centering
	\subfigure{\includegraphics[width=2.35in, height=1.8in]{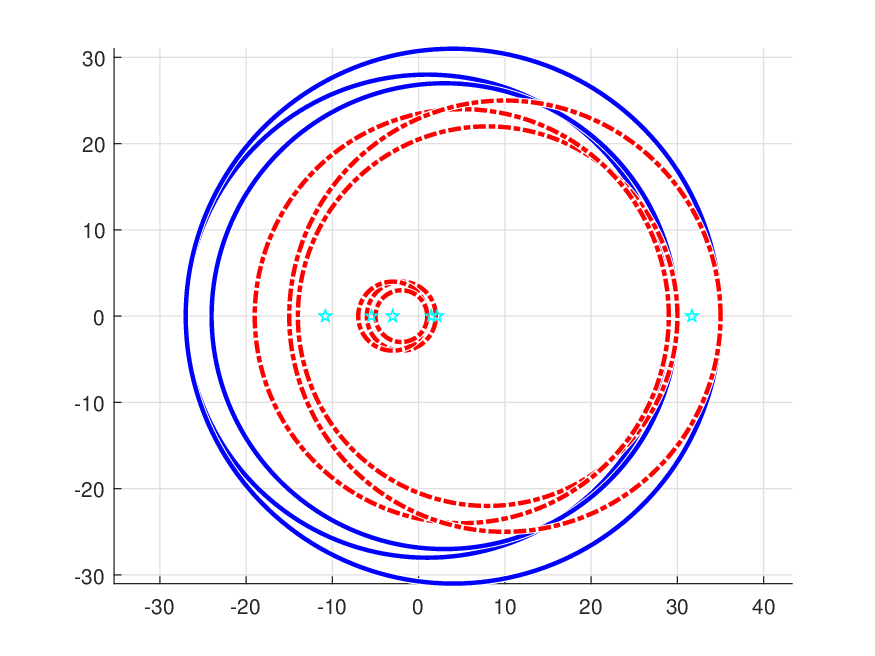}}
	\subfigure{\includegraphics[width=2.35in, height=1.8in]{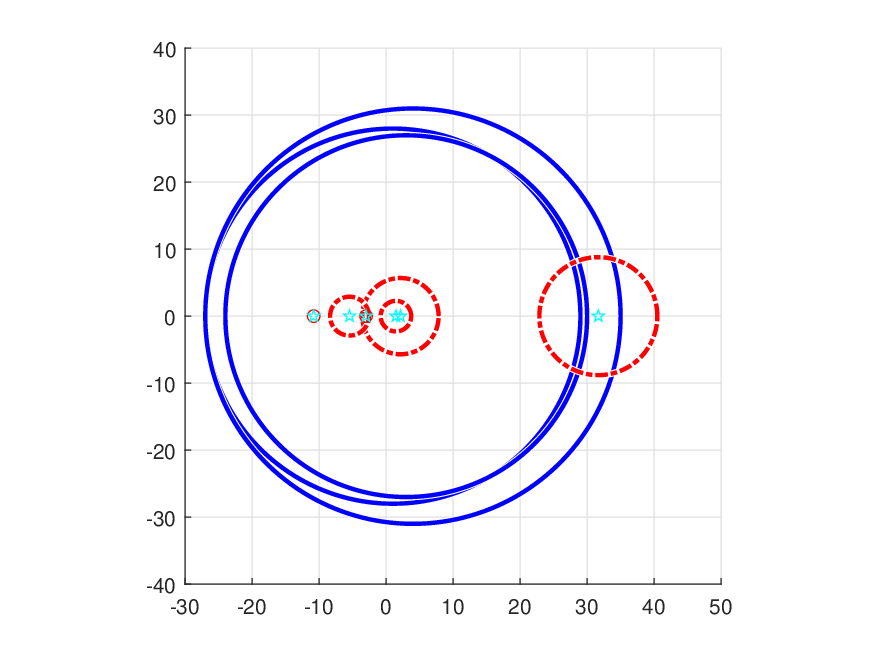}}
	\caption{
		The Gershgorin discs  (represented by the blue solid lines), obtained by Theorem $5.2$ given in  \cite{cao2021tensor}, are compared with the Gershgorin discs derived from  Theorem \ref{GerThm} presented in this paper (represented by the red dash-dot lines) under two similarity transformations for Example  \ref{ExampleCom}.
		(Left: $X = I$; Right: $X$ is a specifically selected transformation)} \label{FigEx3-1}
\end{figure}

\begin{example}\label{ExampleCom}
	Let $\mathcal{A}\in \mathbb{R}^{3 \times 3 \times 2}$.  Its frontal slices are $A_1$ and $A_2$, with their entries 
	$$
	A_1 = \left(
	\begin{array}{ccc}
		3 & 7 & 6 \\
		4 & 1 & 9 \\
		5 & 9 & 4 \\
	\end{array}
	\right), A_2 = \left(
	\begin{array}{ccc}
		5 & 7 & 2 \\
		1 & 4 & 10 \\
		2 & 9 & 6 \\
	\end{array}
	\right).
	$$
	By Theorem $5.2$ given in \cite{cao2021tensor}, we can draw the Gershgorin discs that contain all T-eigenvalues of tensor $\mathcal{A}$. The corresponding results, depicted by blue solid lines, are visually represented  in Fig. \ref{FigEx3-1}.
	
	To illustrate the Gershgorin discs generated by our presented result (cf. Theorem \ref{GerThm}), we employ two distinct similarity transformations. Initially, the matrix $\operatorname{bcirc}(\mathcal{A})$ can be  block-diagonalized by using the normalized DFT matrix. And we choose the transformation matrix $X$ to be the identity matrix, denoted as $X = I_{mn}$, where $m = 3$ and $n = 2$. 
	Then 
	\begin{equation*}
		I_{mn}^{-1} (F_n^{\mathrm{H}} \otimes I_m) \operatorname{bcirc}(\mathcal{A}) 	(F_n \otimes I_m) I_{mn}=D+F :=M= \left(
		\begin{array}{cc}
			M_1 & O  \\
			O & M_2  \\
		\end{array}
		\right),
	\end{equation*}
	where $O$ corresponds to  the $3\times 3$ zero matrix. 
	
	The Gershgorin discs, defined by \eqref{GerDiscs}, are depicted in the left picture of Fig. \ref{FigEx3-1} as red dash-dot circles. Notably, these red dash-dot circles are contained within the blue solid circles, signifying a tighter bound compared to the result presented in \cite{cao2021tensor}.
	Subsequently, we endeavor to seek smaller Gershgorin discs by designing a new transformation $X_{\text{new}}$. We apply Schur triangularization to matrices $M_1$ and $M_2$. 
	Let 
	$
	M_i = S_i T_i S_i^{-1}, i = 1,2.
	$ Set 
	$$
	X_{new} = \left(
	\begin{array}{cc}
		S_1 & O  \\
		O & S_2  \\
	\end{array}
	\right).
	$$
	Then we get
	\begin{equation*}
		X_{new}^{-1} (F_n^{\mathrm{H}} \otimes I_m) \operatorname{bcirc}(\mathcal{A}) 	(F_n \otimes I_m) X_{new}=D_{new}+F_{new},
	\end{equation*}
	in which we adopt the similar notations as presented in Theorem \ref{GerThm}.
	In this case, $F_{\text{new}}$ is an upper quasi-triangular matrix that contains more zero entries than $F$. By applying the criterion \eqref{GerDiscs} once again, we establish new Gershgorin circles, displayed in the right picture of Fig. \ref{FigEx3-1}. Evidently, the new red dash-dot circles are significantly smaller than the blue solid ones.
	
	In both pictures within Fig. \ref{FigEx3-1}, the cyan pentagrams represent true T-eigenvalues of $\mathcal{A}$. This observation highlights that opting for an appropriate transformation matrix can lead to tighter bounds or even yield the true T-eigenvalues.
	The result presented in \cite{cao2021tensor} is characterized by its simplicity and conciseness, which consequently incurs less computational cost. However, this advantage comes at the cost of producing only rough approximations for the bounds. In contrast, our derived result permits the adjustment of the transformation matrix to achieve more precise outcomes, albeit at the expense of increased computational resources.
\end{example}

\begin{remark}
	Eigenvalue sensitivity analysis is a subject of particular interest to many researchers, especially in the context of the symmetric eigenproblem for real matrices, where orthogonal transformations are commonly employed. Notably, classical results such as the Gershgorin disk theorem and the perturbation result Wielandt-Hoffman theorem \cite{Golub2013matrix} for real symmetric matrices  are well-established and widely recognized.  However, it is important to note that these results are not directly applicable to the real third-order tensor case. The primary reason for this limitation is that the block-diagonalized matrix of a real symmetric tensor, obtained via the normalized DFT matrix, may not retain its symmetry property. 
	Consider the symmetric tensor $\mathcal{A}$ with three frontal slices: $A_1 = \begin{bmatrix} 1 & 2 ; 2 & 4 \end{bmatrix}$, $A_2 = \begin{bmatrix} 5 & 6 ; 7 & 8 \end{bmatrix}$, and $A_3 = \begin{bmatrix} 5 & 7 ; 6 & 8 \end{bmatrix}$, represented in MATLAB notation. This serves as a meaningful illustrative example.
\end{remark}

\subsection{Bauer-Fike Theorem for Tensors}

The Bauer-Fike theorem is a classical result for complex-valued diagonalizable matrices, dealing with eigenvalue perturbation theory. It provides an absolute upper bound for the deviation of a perturbed matrix eigenvalue from a chosen eigenvalue of the original matrix, estimated by the product of the condition number of the eigenvector matrix and the perturbation's norm \cite{Bauer-Fike}.

Extending the Bauer-Fike theorem to the case of non-diagonalizable matrices has been explored in \cite{1986Generalization,Golub2013matrix}. 
Moreover,  the perturbation result depending on  part of the spectra of a matrix with disjoint spectral sets was also considered in  \cite{1986Generalization}. 

In this subsection, we first present a generalization of the renowned Bauer-Fike theorem, applicable not only to diagonalizable matrices but also to the third-order tensor case. Furthermore, we extend it into other general forms to encompass non-diagonalizable matrices within the context of third-order tensors.

Before presenting the main results, we introduce the definitions of tensor-$p$ norms, following a similar formulation to the tensor spectral norm employed in tensor robust principal component analysis \cite{lu2019tensor}. Additionally, we define the tensor condition number for simplicity.  
Let $\mathcal{A}\in \mathbb{C}^{m \times m \times n}$. The tensor-$p$ norm    is defined as $\|\mathcal{A}\|_p := \|\operatorname{bcirc}(\mathcal{A})\|_p$, and the condition number of $\mathcal{A}$ is denoted as  $\kappa_{p}(\mathcal{A}) :=\|\operatorname{bcirc}(\mathcal{A})\|_{p}\left\|\operatorname{bcirc}(\mathcal{A}^{-1})\right\|_{p}$.

\begin{theorem}[Bauer-Fike theorem for tensors] \label{Bauer Ten}
	Let  $\mathcal{A}\in \mathbb{C}^{m \times m \times n}$ be an F-diagonalizable tensor. This implies the existence of an invertible tensor $\mathcal{P}$ such that 
	\begin{equation}\label{F-diagonal}
		\mathcal{P}^{-1}*	\mathcal{A} * \mathcal{P}  = \mathcal{D}, 
	\end{equation}
	where $\mathcal{D}$ is an F-diagonal tensor. 
	Assuming $\mu$ is a T-eigenvalue of $\mathcal{A} + \delta \mathcal{A}$, with $\delta$ being a small number and indicating a small perturbation on the original tensor $\mathcal{A}$. Then, considering the spectral norm or Frobenius norm cases, there exists a T-eigenvalue $\lambda$ of $\mathcal{A}$ such that:
	\begin{equation}\label{2Fnorm}
		|\lambda-\mu| \leq \kappa_{p}(\mathcal{P})\|\delta \mathcal{A}\|_{p}, \ p = 2, F.
	\end{equation}
	Moreover, for  the $1$- and $\infty$-norms, we have 
	\begin{equation*}%\label{1infnorm}
		|\lambda-\mu| \leq \kappa_{p}(\mathcal{P})\kappa_{p}(F_n \otimes I_m)\|\delta \mathcal{A}\|_{p}, \ p = 1, \infty.
	\end{equation*}
\end{theorem}

\begin{proof}
	By Lemma \ref{Lemmabcirc}, we know that the equality (\ref{F-diagonal}) is equivalent to 
	$$
	(	\operatorname{bcirc}(\mathcal{P}))^{-1} \operatorname{bcirc}(\mathcal{A}) \operatorname{bcirc}(\mathcal{P}) = \operatorname{bcirc}(\mathcal{D}).
	$$
	Since the right-hand side of the above equality is a block circulant matrix, thus it can be block-diagonalized by the normalized DFT matrix. Then we have
	$$
	(F_n^{\mathrm{H}} \otimes I_m) \operatorname{bcirc}(\mathcal{D}) 	(F_n \otimes I_m) = D = 	\left[\begin{array}{cccc}
		{D^{(1)}} & {} & {} & {} \\
		{} & {D^{(2)}} & {} & {} \\
		{} & {} & {\ddots} & {} \\
		{} & {} & {} & {D^{(n)}} 
	\end{array}\right].
	$$
	By Lemma \ref{lemmaFdiagonal iff}, we can see that $D^{(1)}, D^{(2)}, \ldots, D^{(n)}$ are diagonal matrices since $\mathcal{D}$ is an F-diagonal tensor. Let $X = 	\operatorname{bcirc}(\mathcal{P}) (\kappa_{p}(\mathcal{A})) $. Then we get $X^{-1}\operatorname{bcirc}(\mathcal{A})X=D$. Hence by Bauer-Fike theorem \cite{Golub2013matrix}, there is an eigenvalue $\lambda$ of bcirc(A), hence a T-eigenvalue $\lambda$ of $\mathcal{A}$,  such that 
	\begin{align*}
		|\lambda-\mu| &\leq \|X^{-1}\| \|X\| \|\operatorname{bcirc}(\delta \mathcal{A})\|\\
		& \leq \kappa_{p}(\mathcal{P}) \kappa_{p}(F_n \otimes I_m)\|\delta \mathcal{A}\|_{p}.
	\end{align*}
	Under the $2$-norm case, since $F_n \otimes I_m$ and $F_n^{\mathrm{H}} \otimes I_m$ are unitary, then 
	$$
	\kappa_{2}(F_n\otimes I_m) = \|F_n \otimes I_m\|_2 \|F_n^{\mathrm{H}}  \otimes I_m\|_2 = 1,
	$$
	and we obtain the inequality \eqref{2Fnorm}. 
	The result for the Frobenius norm follows trivially from the definitions presented earlier, as the spectral norm of tensor $\mathcal{A}$ is not greater than its Frobenius norm.
	The proof is completed. \qed
\end{proof}

\begin{remark}
	At the same time,  Theorem 5.3 in \cite{cao2021tensor} has also addressed the Bauer-Fike theorem. The theorem states that under certain conditions, the inequality
	$
	|\lambda - \mu| \leq \|\mathcal{P}^{-1}\|_2\|\mathcal{P}\|_2\|\mathcal{A}-\mathcal{B}\|_2
	$
	holds. It is evident that this is a specific case of the result presented above, applicable only to the $2$-norm case.
\end{remark}

The following conclusion describes the relationship between the variation of T-spectra, i.e., the set of all T-eigenvalues, and the difference of two tensors. We omit the proof here as it can be easily obtained by using the above theorem.

\begin{corollary}\label{corollary21}
	Let  $\mathcal{A}, \mathcal{B}\in \mathbb{C}^{m \times m \times n}$. $\mathcal{A}$ is an F-diagonalizable tensor with decomposition as
	\begin{equation*}
		\mathcal{P}^{-1}*	\mathcal{A} * \mathcal{P}  = \mathcal{D}. 
	\end{equation*}
	Let $s_{\mathcal{A}} (\mathcal{B})$ be the distance of  two T-spectral sets $\Lambda_{\mathcal{A}} = \{\lambda_i\}_{i=1}^{mn}$ and $\Lambda_{\mathcal{B}} = \{\mu_i\}_{i=1}^{mn}$.  That is, 
	$$
	s_{\mathcal{A}} (\mathcal{B}) = \max_{1\leq j \leq mn} \{
	\min_{1\leq i \leq mn} |\lambda_i - \mu_j|\} .
	$$
	Then the following conclusion holds,
	$$
	s_{\mathcal{A}} (\mathcal{B}) \leq \kappa_{2}(\mathcal{P}) \|\mathcal{B} - \mathcal{A}\|_2.
	$$
	Moreover, if $\mathcal{A}$ is a normal tensor, then 
	$$
	s_{\mathcal{A}} (\mathcal{B}) \leq \|\mathcal{B} - \mathcal{A}\|_2.
	$$
\end{corollary}

In general, most third-order tensors are not F-diagonalizable, which implies that \eqref{F-diagonal} is usually not satisfied for a given tensor. Therefore, we cannot obtain an F-diagonal tensor through tensor-tensor multiplication by an invertible tensor.
In such cases, we resort to the following decomposition.

\begin{lemma}{\rm  (\cite{Kilmer2011report,Lund2020,Miao2020T}, T-Schur decomposition).}  \label{T-Schur decomposition}
	Let  $\mathcal{A}\in \mathbb{C}^{m \times m \times n}$.  Then there exists a unitary tensor $\mathcal{Q}$ such that 
	\begin{equation*}%\label{Schur-diagonal}
		\mathcal{Q}^{-1}*	\mathcal{A} * \mathcal{Q}  = \mathcal{T} = \mathcal{D} + \mathcal{N},
	\end{equation*}
	where $\mathcal{D}$ is an F-diagonal tensor, and each frontal slice of $\mathcal{N}$ is strictly upper triangular.
\end{lemma}

In the following, we present two general cases of the Bauer-Fike theorem for tensors (i.e., Theorem \ref{Bauer Ten}). These theorems are based on the T-Schur decomposition and can be regarded as generalizations of the Bauer-Fike theorems for non-diagonalizable matrices given in \cite{1986Generalization,Golub2013matrix}. The detailed proof are provided in the appendix section.

\begin{theorem}\label{Theorem Schur} {\rm (Generalization of Bauer-Fike theorem for tensors I).}
	Let $\mathcal{Q}^{-1}*	\mathcal{A} * \mathcal{Q}  = \mathcal{D} + \mathcal{N}$ be a T-Schur decomposition of $\mathcal{A}\in \mathbb{C}^{m \times m \times n}$. Given $\mathcal{B} \in \mathbb{C}^{m \times m \times n}$ and $\varepsilon$ as a small scalar, suppose $\mu$ is a T-eigenvalue of $\mathcal{A} + \varepsilon \mathcal{B}$, and let $q$ be the smallest positive number such that $|N|^q = 0$, in which
	$$
	N := 	 \left[\begin{array}{llll}
		N^{(1)} & & & \\
		& N^{(2)} & & \\
		& & \ddots & \\
		& & & N^{(n)}
	\end{array}\right] =  (F_n^{\mathrm{H}} \otimes I_m) 	 \operatorname{bcirc}(\mathcal{N}) 
	(F_n \otimes I_m),
	$$
	and $|N| = (|N_{ij}|)$ denotes the absolute of a matrix element-wisely.  Then, for the spectral and Frobenius norms, we establish the following bounds:
	\begin{equation} \label{General BF-theorem}
		\min _{\lambda \in \Lambda(\mathcal{A})}|\lambda-\mu| \leq \max \left\{\theta, \theta^{1 / q}\right\},
	\end{equation}
	in which
	\begin{equation*}
		\theta =  \| \varepsilon \mathcal{B} \|_p \sum_{k=0}^{q-1}\|N\|_p^{k}, \ p = 2, F.
	\end{equation*}
	Additionally, for the $1$- and $\infty$-norms, the bounds are given by:
	\begin{equation} \label{General BF-theorem-p}
		\min _{\lambda \in \Lambda(\mathcal{A})}|\lambda-\mu| \leq \max\{\theta_p, \theta_p^{1/q}\},
	\end{equation}
	where
	$$
	\theta_p = \|\varepsilon \mathcal{B} \|_p \kappa_{p}(\mathcal{Q}) \kappa_{p}(F_n \otimes I_m) \sum_{k=0}^{q-1}\|N\|_2^{k}, \ p = 1, \infty.
	$$
\end{theorem}

We now present a more generalized result of  Theorem \ref{Bauer Ten}. In contrast to the previous theorem, which involves an F-diagonal tensor, the following result considers the block-diagonal case.

Let $\mathcal{A}\in \mathbb{C}^{m\times m\times n}$.	Notice that $\operatorname{bcirc}(\mathcal{A}) $ can be  block-diagonalized as follows, 
$$
(F_n^{\mathrm{H}} \otimes I_m) \operatorname{bcirc}(\mathcal{A}) 	(F_n \otimes I_m)  = 	\left[\begin{array}{cccc}
	{A^{(1)}} & {} & {} & {} \\
	{} & {A^{(2)}} & {} & {} \\
	{} & {} & {\ddots} & {} \\
	{} & {} & {} & {A^{(n)}} 
\end{array}\right].
$$
By Lemma \ref{lemmaFdiagonal iff}, we know that $A^{(i)}$  where  $i = 1, \ldots, n$ may not diagonal since generally tensor $\mathcal{A}$ is  not F-diagonal.
For each matrix $A^{(i)}$, suppose  $X^{(i)}$ is a transformation matrix such that   $(X^{(i)})^{-1}A^{(i)} X^{(i)} = \operatorname{diag}(A^{(i)}_{k_i})$ where $A^{(i)}_{k_i}$ is in triangular Schur form with 
$$
A^{(i)}_{k_i} = D^{(i)}_{k_i} + N^{(i)}_{k_i}, \ k_i = 1, \ldots, \ell_i,
$$
in which $D^{(i)}_{k_i}$ is diagonal for each $i$.
Denote $$ \operatorname{bcirc}(\mathcal{X})  = 	(F_n \otimes I_m) \operatorname{diag}(X^{(1)}, X^{(2)}, \ldots, X^{(n)}) (F_n^{\mathrm{H}} \otimes I_m).$$ Then we get
\begin{equation*} \label{Exp bcircA tilta}
	\begin{aligned}
		& (F_n^{\mathrm{H}} \otimes I_m)  \operatorname{bcirc}(\mathcal{X})^{-1}  \operatorname{bcirc}(\mathcal{A})  \operatorname{bcirc}(\mathcal{X})  
		(F_n \otimes I_m)\\
		= &
		\left[\begin{array}{llll}
			\operatorname{diag}(A^{(1)}_{k_1}) & & & \\
			& \operatorname{diag}(A^{(2)}_{k_2}) & & \\
			& & \ddots & \\
			& & & \operatorname{diag}(A^{(n)}_{k_n})
		\end{array}\right]
		=\left[\begin{array}{ccccccc}
			A_{1}^{(1)} & & & & & & \\
			& \ddots & & & & & \\
			& & A_{\ell_{1}}^{(1)} & & & & \\
			& & & \ddots & & & \\
			& & & & A_{1}^{(n)} & & \\
			& & & & & \ddots & \\
			& & & & & & A_{\ell_{n}}^{(n)}
		\end{array}\right]
	\end{aligned}
\end{equation*}
\begin{equation*} 
	\begin{aligned}
		= & \left[\begin{array}{ccccccc}
			D_{1}^{(1)} & & & & & & \\
			& \ddots & & & & & \\
			& & D_{\ell_{1}}^{(1)} & & & & \\
			& & & \ddots & & & \\
			& & & & D_{1}^{(n)} & & \\
			& & & & & \ddots & \\
			& & & & & & D_{\ell_{n}}^{(n)}
		\end{array}\right]
		+
		\left[\begin{array}{ccccccc}
			N_{1}^{(1)} & & & & & & \\
			& \ddots & & & & & \\
			& & N_{\ell_{1}}^{(1)} & & & & \\
			& & & \ddots & & & \\
			& & & & N_{1}^{(n)} & & \\
			& & & & & \ddots & \\
			& & & & & & N_{\ell_{n}}^{(n)}
		\end{array}\right]\\
		:= & \tilde{D} + \tilde{N}.
	\end{aligned}
\end{equation*}

By the above analysis, we have the following conclusion.

\begin{theorem} \label{General two}{\rm (Generalization of Bauer-Fike theorem for tensors II).}
	If $\mu$ is a T-eigenvalue of $\mathcal{A} + \varepsilon \mathcal{B}$ and $q$ is the dimension of
	$A_{k_i}^{(i)}$,  then we have
	\begin{equation*}
		\min _{\lambda \in \Lambda(\mathcal{A})}|\lambda-\mu| \leq \max \left\{\theta_1, \theta_1^{1 / q}\right\},
	\end{equation*}
	where
	\begin{equation*}
		\theta_1 = C\varepsilon \|  \mathcal{B}\|_p \kappa_p(\mathcal{X}), \ p = 2, F.
	\end{equation*}
	and $C = \sum_{k=0}^{q-1}\|N^i\|_2^{k}$ provided that $\max _{j}\left\|\left(A^{(i)}_j-\mu I\right)^{-1}\right\|$ occurring at $j = k_i$.
	
	For  $1$- and $\infty$-norms, under the above condition, we get 
	$$
	\min _{\lambda \in \Lambda(\mathcal{A})}|\lambda-\mu| \leq \max\{\theta_2, \theta_2^{1/q}\},
	$$
	where
	$$
	\theta_2 = C\varepsilon \|  \mathcal{B} \|_p \kappa_p(\mathcal{X}) \kappa_{p}(F_n \otimes I_m),\ p = 1, \infty.
	$$
	
\end{theorem}

\subsection{Kahan Theorem for Tensors}
The perturbation of a Hermite tensor by another Hermite tensor is studied in \cite[Theorem 4.1]{Jin2020}. In the following, we present a result involving the perturbation of a Hermite tensor by any tensors.

\begin{theorem} {\rm(Kahan theorem for tensors).} \label{KahanThm}
	Let  $\mathcal{A} \in \mathbb{R}^{m \times m \times n}$  be a Hermite tensor.  Suppose that  its T-spectral set is denoted as $\Lambda_{\mathcal{A}} = \{\lambda_i\}_{i=1}^{mn}$  such that its T-eigenvalues are arranged in a non-increasing order:
	\begin{equation*}
		\lambda_{\max }=\lambda_{1} \geq \lambda_{2} \geq \cdots \geq \lambda_{mn-1} \geq \lambda_{mn}=\lambda_{\min }.
	\end{equation*}
	Suppose that $\mathcal{B} = \mathcal{A}   + \mathcal{E} $ and let $\Lambda_{\mathcal{B}} = \{\beta_k + i\gamma_k\}_{k=1}^{mn}$  such that 
	$	\beta_{1} \geq \beta_{2} \geq \cdots \geq \beta_{mn-1} \geq \beta_{mn}.
	$
	Let  $$E_y = \frac{\operatorname{bcirc}(\mathcal{E})-\operatorname{bcirc}(\mathcal{E})^{\mathrm{H}}}{2i} $$
	%$$
	%\mathcal{E}_x = \frac{\mathcal{E}+\mathcal{E}^{\mathrm{H}}}{2} \quad \text{and} \quad 
	%\mathcal{E}_y = \frac{\mathcal{E}-\mathcal{E}^{\mathrm{H}}}{2i} 
	%$$
	and
	$$
	\sigma_k = \{  \beta + i\gamma \in \mathbb{C}: |\beta + i\gamma - \lambda_k| \leq \|\mathcal{E}\|_2, |\gamma| \leq  \|E_y\|_2    \}.
	$$
	Then 
	$$
	\Lambda_{\mathcal{B}} \subset  \bigcup_{k=1}^{mn}	\sigma_k.
	$$
\end{theorem}   

\begin{proof}
	On one hand, as stated in Lemma \ref{normal diagonal}, a Hermite tensor can be F-diagonalized by a unitary tensor. According to Corollary \ref{corollary21}, there exists a T-eigenvalue $\lambda_k$ such that $|\beta + i\gamma - \lambda_k| \leq \|\mathcal{E}\|_2$ for any given T-eigenvalue $\beta + i\gamma$ of $\mathcal{B}$.
	
	On the other hand, suppose that $\mathcal{B}*\mathcal{X} = (\beta + i\gamma) \mathcal{X}$, then by the definition of tensor-tensor multiplication, it is equivalent to 	
	$$
	\operatorname{bcirc}(\mathcal{B}) \operatorname{unfold}(\mathcal{X})  = (\beta + i\gamma) \operatorname{unfold}(\mathcal{X}).
	$$
	Thus, it is reasonable to assume that  $\|\operatorname{unfold}(\mathcal{X}) \|_2 = 1$, since $\mathcal{X}$ is nonzero, which implies that the vector $\operatorname{unfold}(\mathcal{X})$ is also nonzero.
	Therefore,
	$$
	( \operatorname{unfold}(\mathcal{X}))^{\mathrm{H}}\operatorname{bcirc}(\mathcal{B}) \operatorname{unfold}(\mathcal{X})  = \beta + i\gamma$$
	and
	 $$
	( \operatorname{unfold}(\mathcal{X}))^{\mathrm{H}}\operatorname{bcirc}(\mathcal{B})^{\mathrm{H}} \operatorname{unfold}(\mathcal{X})  = \beta - i\gamma,
	$$
	which implies that
	$$
	\gamma = \frac{( \operatorname{unfold}(\mathcal{X}))^{\mathrm{H}}[\operatorname{bcirc}(\mathcal{B}) -\operatorname{bcirc}(\mathcal{B})^{\mathrm{H}}  ] \operatorname{unfold}(\mathcal{X}) }{2i} = \operatorname{unfold}(\mathcal{X}))^{\mathrm{H}} E_y \operatorname{unfold}(\mathcal{X}).
	$$
	Thus $|\gamma| \leq  \|E_y\|_2$. 
	
	The conclusion can be obtained from the combination of these two  parts. \qed
\end{proof}

\section{Pseudospectra  of Third-Order Tensors}\label{Pseudospectra sec}
The pseudospectra of finite-dimensional matrices and their extension to linear operators in Banach space have been extensively investigated and summarized in the classical book by Trefethen and Embree \cite{trefethen2005spectra}. In the book, four different definitions of matrix pseudospectra are introduced and shown to be equivalent under certain conditions. The properties of pseudospectra are also discussed, along with a characterization of the pseudospectra for normal matrices. Additionally, for diagonalizable but not necessarily normal matrices, the corresponding Bauer-Fike theorem is presented, which can be found in \cite[Theorems 2.2, 2.3, and 2.4]{trefethen2005spectra}. The book also covers various methods for computing matrix pseudospectra; for more details, refer to \cite[Section 39]{trefethen2005spectra}. Notably, it also includes representative numerical results, including plots of pseudospectra that span nearly every section of the book.

Before presenting our own results, we review the definitions of pseudospectra for the matrix case.

\begin{definition}{\rm (\cite[Section 2]{trefethen2005spectra}).}\label{trepseu}
	Let $A\in\mathbb{C}^{m\times m}$ and $\varepsilon>0$ be arbitrary. $\lambda(A)$ is the set of  eigenvalues of $A$, and the convention $\|(zI-A)^{-1}\|=\infty$ for $z\in\lambda(A)$ is employed.\\
	{\rm (I)} The $\varepsilon$-pseudospectra $\lambda_{\varepsilon}(A)$ of $A$
	is the set of $z\in \mathbb{C}$ such that $\|(zI-A)^{-1}\|>\varepsilon^{-1}$.\\
	{\rm (II)} The $\varepsilon$-pseudospectra $\lambda_{\varepsilon}(A)$ of $A$
	is the set of $z\in \mathbb{C}$ such that $z\in \lambda(A+E)$ for some $E\in \mathbb{C}^{m\times m}$ with $\|E\| < \varepsilon$.\\
	{\rm (III)} The $\varepsilon$-pseudospectra $\lambda_{\varepsilon}(A)$ of $A$
	is the set of $z\in \mathbb{C}$ such that $\|(zI-A)\mathbf{v}\|<\varepsilon$ for some $v\in \mathbb{C}^m$ with $\|\mathbf{v}\| = 1$.\\
	{\rm (IV)} For $\|\cdot\| = \|\cdot\|_2$, $\lambda_{\varepsilon}(A)$ is the set of $z\in \mathbb{C}$ such that $\sigma_{min}(zI-A) <\varepsilon$ where $\sigma_{min}(\cdot)$ denotes the minimum singular value. 
\end{definition}

\begin{remark}\label{rem41}
	For any matrix $A \in \mathbb{C}^{m \times m}$, the equivalence of the three definitions (I), (II), and (III) given in Definition \ref{trepseu} has been established. Moreover, when the spectral norm is chosen, definition (IV) is shown to be equivalent to the other three characterizations of matrix pseudospectra \cite{trefethen2005spectra}.
	
\end{remark}

In this section, we delve into the study of pseudospectra for third-order tensors within the tensor-tensor multiplication framework. Specifically, we explore different formulations of pseudospectra for third-order tensors in Subsection \ref{Sec41}. Subsection \ref{Sec42} is dedicated to the examination of various properties of pseudospectra related to third-order tensors. Finally, in Subsection \ref{Sec43}, we present plots illustrating the computed results of $\varepsilon$-pseudospectra for given tensors  and give an example to illustrate the application for seeking more T-positive definite tensors by using the tensor $\varepsilon$-pseudospectra theory.

\subsection{Pseudospectra  of Third-Order Tensors under Tensor-Tensor Multiplication}\label{Sec41}

Various definitions on $\varepsilon$-pseudospectra of  $m \times m \times n$ tensors are studied in this subsection. 
If the norm $\|\cdot\|$ is not specified, we adopt the convention that it corresponds to the $p$ norm, where $p = 1, 2, \infty.$

\begin{definition}[$\varepsilon$-pseudospectra of  tensors]\label{Definition pseudospectra}
	Let $\mathcal{A} \in \mathbb{C}^{m \times m \times n}$.  Then the block circulant matrix $\operatorname{bcirc}(\mathcal{A})$ generated by  tensor  $\mathcal{A}$ can be factored as follows, 
	\begin{equation} \label{Factorization of bcirc(A)}
		\operatorname{bcirc}(\mathcal{A})
		=
		\left(F_{n} \otimes I_{m}\right) 
		\left[\begin{array}{cccc}
			{A^{(1)}} & {} & {} & {} \\
			{} & {A^{(2)}} & {} & {} \\
			{} & {} & {\ddots} & {} \\
			{} & {} & {} & {A^{(n)}}
		\end{array}\right]
		\left(F_{n}^{\mathrm{H}}\otimes I_{m}\right). 
	\end{equation}
	For each  matrix $A^{(i)}\in\mathbb{C}^{m\times m}$ where $i \in [n] :=\{ 1, \ldots, n\}$, we have
	$$
	\lambda_{\varepsilon}(A^{(i)}):= \left\{z \in \mathbb{C}:\left\|(z I_m-A^{(i)})^{-1}\right\| > \varepsilon^{-1}\right\},
	$$
	where $\varepsilon$ is an arbitrary positive scalar. 	If there are some $i$ such that $zI_m-A^{(i)}$ are singular for $z\in\lambda(A^{(i)})$, we define $\left\|(z I_m-A^{(i)})^{-1}\right\|=\infty$.
	
	In this definition, we denote 	the block diagonal matrix in \eqref{Factorization of bcirc(A)} as
	$\operatorname{diag}(A^{(1)}, \ldots, A^{(n)}) :=A$ and its spectrum as $\Lambda(A)$.
	
	{\rm (I)}	We call
	$$
	\Lambda_{\varepsilon}(\mathcal{A}):= \left\{z \in \mathbb{C}: \max_{i\in[n]}\left\|(z I_m-A^{(i)})^{-1}\right\| > \varepsilon^{-1}\right\}
	$$
	as the $\varepsilon$-pseudospectra of  tensor $\mathcal{A}$.
	
	{\rm (II)}  $\Lambda_{\varepsilon}(\mathcal{A})=\{z \in \mathbb{C}: z \in \Lambda(A+E)$ for some  $E \in \mathbb{C}^{mn \times mn}$ with $\|E\| < \varepsilon\}.$
	
	{\rm (III)}  $	
	\Lambda_{\varepsilon}(\mathcal{A})=\left\{z \in \mathbb{C}:\right.$ There exists $v \in \mathbb{C}^{mn}$ with $\|v\|=1$ such that $\left.\|(A-z I_{mn}) v\| < \varepsilon\right\}.$
\end{definition}

\begin{theorem}
	The three different characterizations (I), (II),  and (III) on $\varepsilon$-pseudospectra of  third-order tensors 
	are equivalent.
\end{theorem}   
\begin{proof}
	For a block diagonal matrix, we find that $\|A\| = \max_{i\in[n]} \|A^{(i)}\|$, and the inverse of $A$ can be obtained by computing the inverse of each $A^{(i)}$. Therefore,  we have 
	$$
	\max_{i\in[n]}\left\|(z I_m-A^{(i)})^{-1}\right\| = \left\|(z I_{mn}-A)^{-1}\right\|,
	$$
	and consequently,
	\begin{equation*}%\label{Def pseu-1}
		\Lambda_{\varepsilon}(\mathcal{A})=\left\{z \in \mathbb{C}:\left\|(z I_{mn}-A)^{-1}\right\| > \varepsilon^{-1}\right\}.
	\end{equation*}
	The equivalence of the above equation with (II) and (III) can be easily deduced from the claims in Remark \ref{rem41}. \qed
\end{proof}

\begin{remark}\label{Remark4-1}
	Notice that $z\in \mathbb{C}$ and it is a variable,  while $A^{(i)}$ is stationary when the tensor $\mathcal{A}$ is given.  Therefore, by (I) in Definition \ref{Definition pseudospectra}, we can also see that 
	$$
	\Lambda_{\varepsilon}(\mathcal{A})= \bigcup_{i=1}^{n}	\Lambda_{\varepsilon}(A^{(i)}) =  \bigcup_{i=1}^{n} \left\{z \in \mathbb{C}:\left\|(z I_m-A^{(i)})^{-1}\right\| > \varepsilon^{-1}\right\}.
	$$
\end{remark}

In the case of the spectral norm, we provide the following definition.
\begin{definition}\label{Definition pseudospectra-singularValue}
	Let $\mathcal{A} \in \mathbb{C}^{m \times m \times n}$.  Then the block-circulant matrix $\operatorname{bcirc}(\mathcal{A})$ generated by the tensor  $\mathcal{A}$ can be factored as shown in (\ref{Factorization of bcirc(A)}). For each square matrix $A^{(i)}$, where $i \in [n]$, we have
	$$
	\Lambda_{\varepsilon}(A^{(i)})=\left\{z \in \mathbb{C}: \sigma_{\min }(z I_m-A^{(i)}) \leq \varepsilon\right\},
	$$
	where $\varepsilon$ is a positive scalar and $\sigma_{min}(\cdot)$ denotes the minimum singular value.
	We refer to 
	$$
	\Lambda_{\varepsilon}(\mathcal{A})= \bigcup_{i=1}^{n}	\Lambda_{\varepsilon}(A^{(i)})
	$$ 	
	as the $\varepsilon$-pseudospectra of  tensor $\mathcal{A}$.
\end{definition}

By Remark \ref{Remark4-1}, we get the following conclusion.
\begin{theorem}
	In the case of the spectral norm, the characterization given in Definition \ref{Definition pseudospectra-singularValue} is equivalent to the three characterizations (I), (II), and (III) given in Definition \ref{Definition pseudospectra}.
\end{theorem}

\subsection{Properties of Pseudospectra of Tensors}\label{Sec42}
In this subsection, we investigate the properties of pseudospectra for third-order tensors. Several fundamental results are presented in the following theorem.
\begin{theorem} \label{ProPseu}
	Let $\mathcal{A} \in \mathbb{C}^{m \times m \times n}$ and consider an arbitrary positive scalar $\varepsilon$. The following properties hold for the pseudospectra $\Lambda_{\varepsilon}(\mathcal{A})$:
	
	{\rm (I)} The set $\Lambda_{\varepsilon}(\mathcal{A})$  is nonempty, open, and bounded. Moreover, there are at most $nm$ connected components,   each containing one or more T-eigenvalues of $\mathcal{A}$.
	
	{\rm (II)} For any $c \in \mathbb{C}$, we have  $\Lambda_{\varepsilon}(\mathcal{A}+c)=c+\Lambda_{\varepsilon}(\mathcal{A})$, where $\mathcal{A}+c$ is shorthand for $\mathcal{A}+c\mathcal{E}$,  and $\mathcal{E}$ is the identity tensor with the same size as $\mathcal{A}$.
	
	{\rm (III)}  For any nonzero $c \in \mathbb{C}$, we have $\Lambda_{|c| \varepsilon}(c \mathcal{A})=c \Lambda_{\varepsilon}(\mathcal{A})$. 
	
	{\rm (IV)}  If the spectral norm is applied, then $\Lambda_{\varepsilon}\left(\mathcal{A}^{\mathrm{H}}\right)=\overline{\Lambda_{\varepsilon}(\mathcal{A})}$.
	
\end{theorem}

\begin{remark}
	
	By the results above, we can see that the function of pseudospectra on tensor $\mathcal{A}\in \mathbb{C}^{m \times m \times n}$ is linear.
\end{remark}

The properties of  pseudospectra on normal tensors are given next. 
\begin{theorem}\label{PseuNormal}{\rm (Pseudospectra of a normal tensor).}
	Let  $\Delta_{\varepsilon}$ be an open $\varepsilon$-ball.  That is, 
	$
	\Delta_{\varepsilon}=\{z \in \mathbb{C}:|z|<\varepsilon\}.
	$
	A sum of sets is defined as 
	$$
	\Lambda(\mathbf{\mathcal{A}})+\Delta_{\varepsilon}=\left\{z: z=z_{1}+z_{2}, z_{1} \in \Lambda(\mathbf{\mathcal{A}}), z_{2} \in \Delta_{\varepsilon}\right\},
	$$
	where $\Lambda(\mathbf{\mathcal{A}})$ is the T-spectra (sets of  T-eigenvalues) of the tensor $\mathcal{A}$.
	Then	for any tensor $\mathcal{A}\in \mathbb{C}^{m\times m \times n}$,
	we have 
	\begin{equation}\label{notnormal-subset}
		\Lambda_{\varepsilon}(\mathcal{A}) \supseteq \Lambda(\mathcal{A})+\Delta_{\varepsilon} \quad \forall \varepsilon>0.
	\end{equation}
	Moreover, if the tensor $\mathcal{A}$ is normal and $\|\cdot\|=\|\cdot\|_{2}$, then
	\begin{equation*}%\label{normal-equal}
		\Lambda_{\varepsilon}(\mathcal{A}) = \Lambda(\mathcal{A})+\Delta_{\varepsilon} \quad \forall \varepsilon>0.
	\end{equation*}
	%Conversely, if $\|\cdot\|=\|\cdot\|_{2}$, the above equality (\ref{normal-equal}) implies that the tensor $\mathcal{A}$ is normal.
\end{theorem}

\begin{theorem}{\rm (Bauer-Fike theorem).}
	Suppose tensor $\mathcal{A} \in \mathbb{C}^{m\times m \times n}$ is F-diagonalizable, i.e., it has the decomposition \eqref{F-diagonal}.  If the spectral norm is applied, then for each positive scalar $\varepsilon$, we have
	$$
	\Lambda(\mathcal{A}) + \Delta_{\varepsilon} \subseteq     \Lambda_{\varepsilon}(\mathcal{A})
	\subseteq    \Lambda(\mathcal{A}) + \Delta_{\varepsilon \kappa_{2}(\mathcal{P})}.
	$$
\end{theorem}   
\begin{proof}
	We only need to prove the second inclusion. By the definition of $\varepsilon$-pseudospectra,  Lemma \ref{Lemmabcirc} and decompositions \eqref{F-diagonal} and \eqref{diagD}, it is not hard to find that
	\begin{equation*}
		\begin{aligned}
			\|(zI_{mn} -  A)^{-1}\| 
			&= \| \operatorname{bcirc}(\mathcal{P}) (zI_{mn} - \operatorname{bcirc}(\mathcal{D}))^{-1} \operatorname{bcirc}(\mathcal{P})^{-1}\| \\
			&\leq \frac{\kappa(\mathcal{P})}{\operatorname{dist}(z, \Lambda(\operatorname{bcirc}(\mathcal{D})))} = \frac{\kappa(\mathcal{P})}{\operatorname{dist}(z, \Lambda(D))}\\
			& = \frac{\kappa(\mathcal{P})}{\operatorname{dist}(z, \Lambda(\mathcal{A}))}.
		\end{aligned}
	\end{equation*}
	By (I) of  Definition \ref{Definition pseudospectra}, we get our conclusion. \qed
	
\end{proof}

We next present a result, which is a generalization of the Gershgorin disc theorem for tensors given in Theorem \ref{GerThm} under the condition $X$ is the identity matrix, 
 to  locate the $\varepsilon$-pseudospectra for T-eigenvalues of tensor $\mathcal{A} \in \mathbb{C}^{m \times m \times n}$.
  Let $\lambda \in 	\Lambda_{\varepsilon}(\mathcal{A})$. By (II) in Definition \ref{Definition pseudospectra} there exists a matrix $E \in \mathbb{C}^{mn \times mn}$ with $\|E\| < \varepsilon$ and a nonzero vector $\mathbf{y}\in \mathbb{C}^{mn}$ such that $(A+E)\mathbf{y}=\lambda \mathbf{y}$, then by similar proof given in  \cite{KostiPseudospectra2016} we obtain the following conclusion. 
\begin{theorem}[$\varepsilon$-pseudospectra Gershgorin sets]\label{Pseu_Loca}
	Let  $\mathcal{A} \in \mathbb{C}^{m \times m \times n}$ be an arbitrary tensor and $\varepsilon\geq 0$. Then 
	\begin{equation*}
		\Lambda_{\varepsilon}(\mathcal{A}) \subseteq \Gamma_{\varepsilon}(\mathcal{A}):=\bigcup_{i \in S}\left\{z \in \mathbb{C}:\left|a_{i i}-z\right| \leqslant R_i(A)+\varepsilon\right\}
	\end{equation*}
	where 
	$A = 	\left(F_{n}^{\mathrm{H}}\otimes I_{m}\right) \operatorname{bcirc}(\mathcal{A})
	\left(F_{n}\otimes I_{m}\right)$, $a_{ii}$ is the diagonal entry of $A$ and $R_i(A):=\sum_{j \in S \backslash\{i\}}\left|a_{i j}\right|$ for all $i\in S:=\{1,2,\cdots, mn\}$.
\end{theorem}

\begin{remark}
Pseudospectra localizations for matrices or (generalized) tensor eigenvalues \cite{qi2005eigenvalues} have several applications, including the search for more positive definite tensors \cite{LiLiuWei2019} and testing the stability of dynamical systems \cite{KostiPseudospectra2016}. The study of pseudospectra for T-eigenvalues of tensors will contribute to the research on T-positive definiteness of third-order tensors and T-semidefinite programming \cite{zheng2020t}.
Recently, a new class of multilinear time-invariant systems, dependent on the Einstein product, has been studied \cite{Chen2021}. In these systems, the states, inputs, and outputs are tensors. It would be interesting to explore multilinear time-invariant systems that depend on tensor-tensor multiplication and investigate their applications,  leaving this as a topic for future research.
\end{remark}

\subsection{Examples of the $\varepsilon$-Pseudospectra for Tensors}\label{Sec43}
The study of spectra and pseudospectra in matrix cases indicates that while eigenvalues are successful tools for solving mathematical problems in various fields, they may not always provide a satisfactory answer for questions that mainly depend on the spectra, especially for nonnormal matrices. As an alternative, pseudospectra attempt to provide approximate solutions by offering reasonably tight bounds and engaging geometric interpretations \cite{trefethen2005spectra}. Given the significance of pseudospectra in solving matrix problems, we aim to extend this tool to tensors based on the theoretical analysis in Subsection \ref{Sec41}.

In this subsection, we utilize the pseudospectra tool to approximate T-eigenvalues of third-order tensors. Several examples, adapted from the matrix case in \cite[Section 3]{trefethen2005spectra}, are provided with spectra and pseudospectra characterization plots to illustrate its utility.

\begin{figure}[h]
	\centering
	\includegraphics[width=\textwidth]{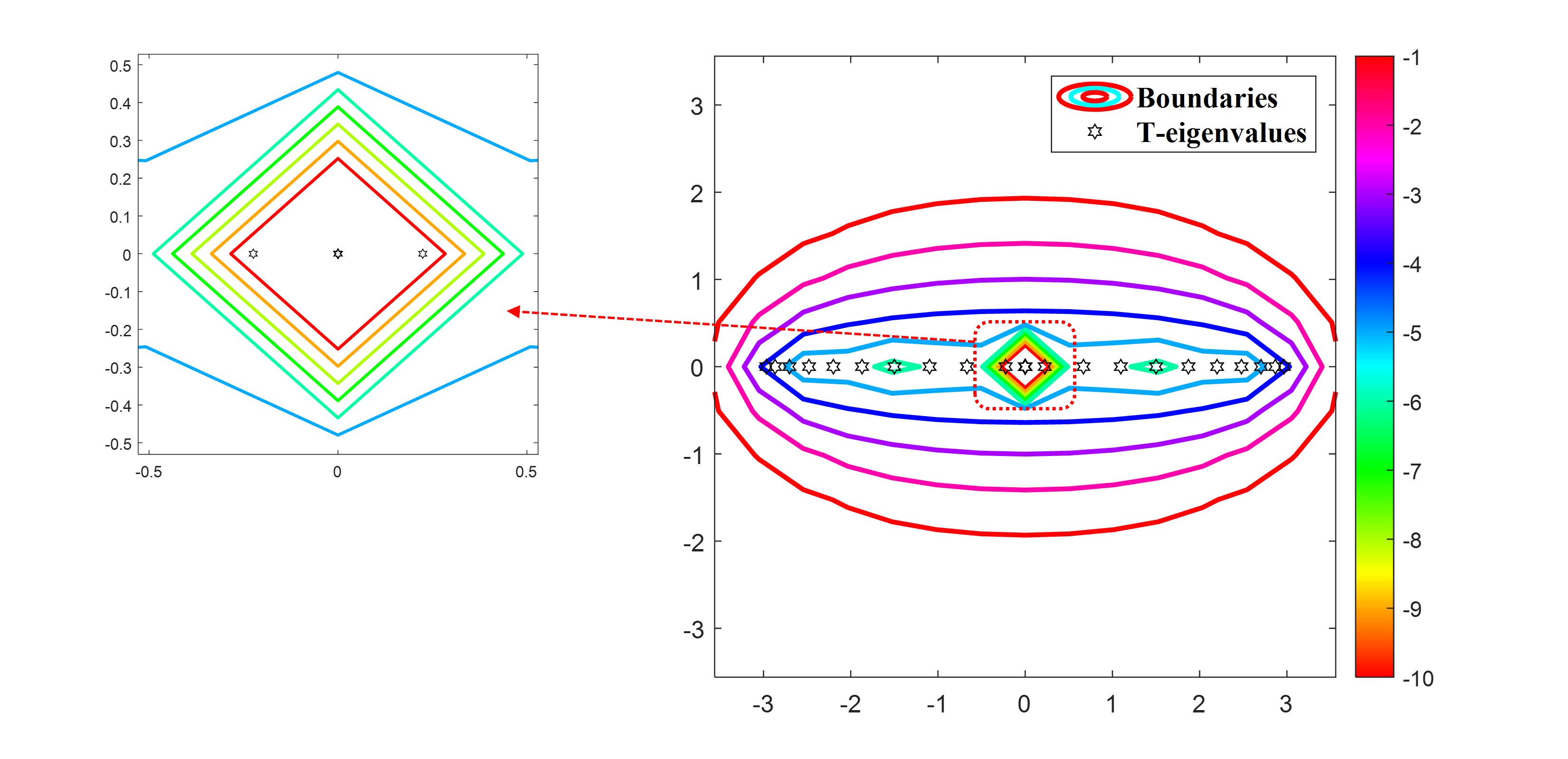}
	\caption{Boundaries of pseudospectra $\Lambda_{\varepsilon}(\mathcal{A})$ for tensor $\mathcal{A}\in \mathbb{R}^{20\times 20\times 3}$ with frontal slices $A_1 = A_2=A_3 = T_{pz}$ for different $\varepsilon=10^{-1}, 10^{-2}, \ldots, 10^{-10}$. The T-eigenvalues are plotted as hexagons.} \label{figure_sym0}
\end{figure}

\begin{example}
	Let $\mathcal{A}$ be a third-order tensor with three frontal faces $A_1, A_2$ and $A_3$. First, we consider an example that all frontal faces are the same and each frontal slice is a tridiagonal Toeplitz matrix.  That is, $A_1 = A_2=A_3 = T_{pz}$ with entries given as follows,
	\begin{equation*}
		T_{pz} =\left(\begin{array}{ccccc}
			0 & 1 & & & \\
			\frac{1}{4} & 0 & 1 & & \\
			& \ddots & \ddots & \ddots & \\
			& & \frac{1}{4} & 0 & 1 \\
			& & & \frac{1}{4} & 0
		\end{array}\right) \in \mathbb{R}^{N \times N}.
	\end{equation*}
\end{example}

We can see that 
the size of $\operatorname{bcirc}(\mathcal{A})$ is  $3N \times 3N$. However, it is non-symmetrical. Notice that $T_{pz}$ can be symmetrized by a diagonal similarity transformation, that is, 
$$D T_{pz} D^{-1}= S,$$ where $D=\operatorname{diag}\left([2,4, \ldots, 2^{N}]\right)$
and 
\begin{equation*}
	S=\left(\begin{array}{ccccc}
		0 & \frac{1}{2} & & & \\
		\frac{1}{2} & 0 & \frac{1}{2} & & \\
		& \ddots & \ddots & \ddots & \\
		& & \frac{1}{2} & 0 & \frac{1}{2} \\
		& & & \frac{1}{2} & 0
	\end{array}\right)\in \mathbb{R}^{N \times N}.
\end{equation*}

A tensor $\mathcal{D} \in \mathbb{R}^{N\times N\times 3}$ with frontal faces $D_1, D_2$ and $D_3$ is constructed as follows. That is,  we let $D_1 = D$ and choose $D_2$ and $D_3$ as the zero matrices $O_{N \times N}$. Then another tensor with three frontal faces $D^{-1}$, $O_{N\times N}$, and $O_{N\times N}$ would be the inverse of $\mathcal{D}$. By applying the tensor-tensor multiplication, we get
$$
\mathcal{D}*\mathcal{A}*\mathcal{D}^{-1} = \mathcal{S}.
$$
It is not hard to calculate that each frontal slice of tensor $\mathcal{S}$ is exactly the symmetric matrix $S$. Then  we know that all  T-eigenvalues of $\mathcal{A}$ are the same as those T-eigenvalues of $\mathcal{S}$ which are all real and only appear in the real axis since  $\operatorname{bcirc}(\mathcal{S})$ is symmetric.

However, the  pseudospectra of $\mathcal{A}$  lie a little far from the real axis. We consider a case where $N = 20$ and the results are given in Fig. \ref{figure_sym0}  in which the T-eigenvalues of tensor $\mathcal{A}$ are plotted as hexagons. The spectral norm is chosen here and we set $\varepsilon=10^{-1}, 10^{-2}, \ldots, 10^{-10}$.  By the numerical results, we can see that the boundaries of  $\Lambda_{\varepsilon}(\mathcal{A})$ may lie far beyond the real axis if  $\varepsilon$ is too big.

\begin{figure}[h]
	\centering
	\includegraphics[width=\textwidth]{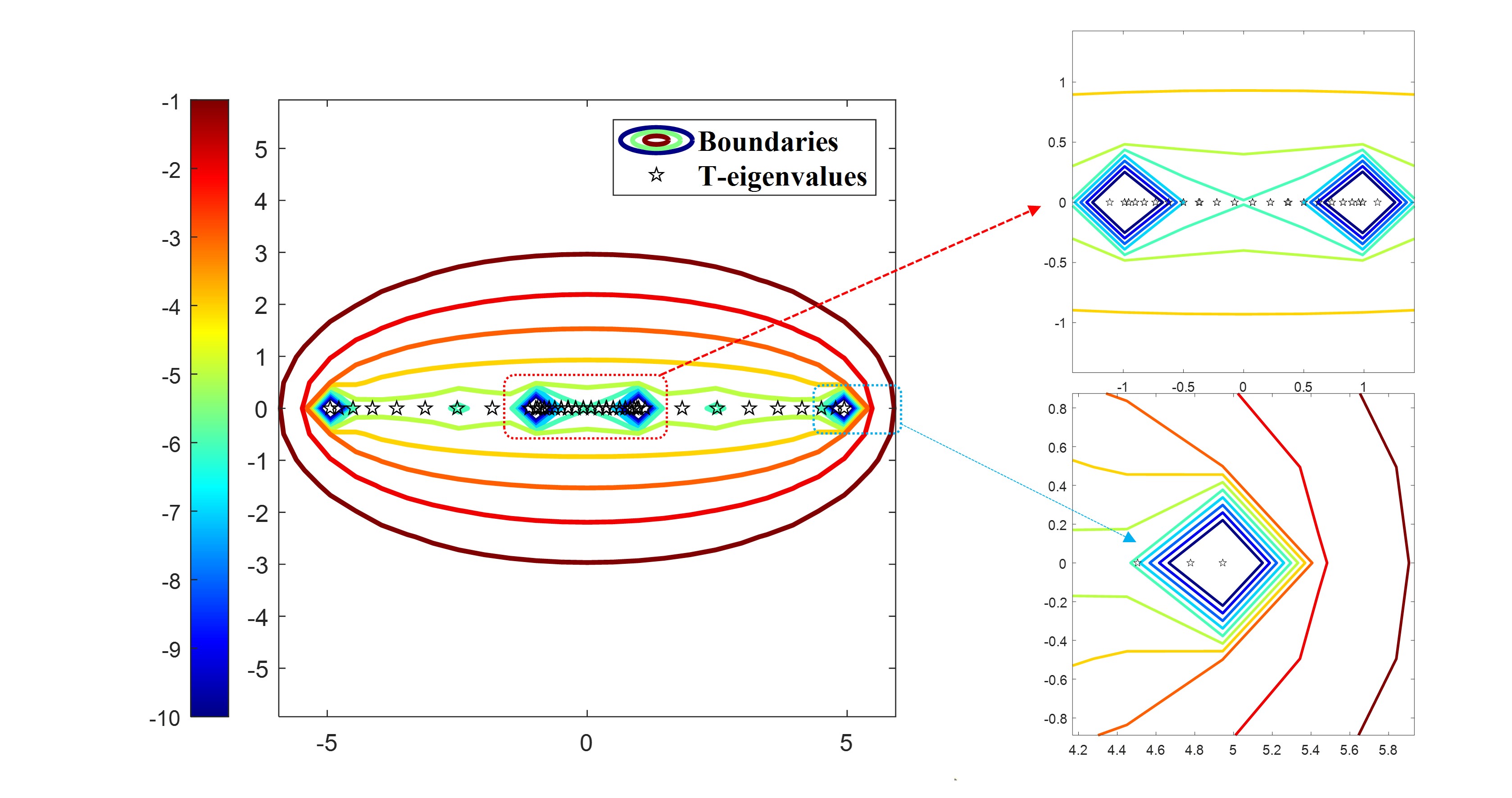}
	\caption{Boundaries of pseudospectra $\Lambda_{\varepsilon}(\mathcal{B})$ for tensor $\mathcal{B}\in \mathbb{R}^{20\times 20\times 3}$ with frontal slices $B_1 = T_{pz}$ and $ B_2=B_3 =  2 T_{pz}$
		for different $\varepsilon=10^{-1}, 10^{-2}, \ldots, 10^{-10}$.  The T-eigenvalues of tensor $\mathcal{B}$ are plotted as pentagons. } \label{figure_sym00}
\end{figure}

\begin{example}
	Many T-eigenvalues of  tensor $\mathcal{A}$ considered in the last example are zero. Another tensor $\mathcal{B}$ with frontal slices satisfying $B_1 = T_{pz}$ and $ B_2=B_3 =  2 T_{pz}$ is considered next. The corresponding matrix $ \operatorname{bcirc}(\mathcal{B})$ is non-symmetrical but all its eigenvalues are  real.  Similar numerical results on boundaries of pseudospectra $\Lambda_{\varepsilon}(\mathcal{B})$ 
	are given in Fig. \ref{figure_sym00}.
\end{example}

Tensors with all real T-eigenvalues are rare generally.  In the following, we consider
two cases that tensors have complex T-eigenvalues.

\begin{example}
	Suppose  $\mathcal{C}\in \mathbb{R}^{20\times 20\times 3}$ is a third-order tensor with frontal slices $C_1,C_2$ and $C_3$ such that $C_3 = 2C_2 = 4C_1 = 4T_{pz}$. By Definition \ref{def3-1} and using the $\operatorname{bcirc}(\cdot)$ operator,  T-eigenvalues which may complex are plotted as crosses `$\times$' in Fig. \ref{figure_sym000}. To approximate T-spectra of tensor $\mathcal{C}$, boundaries of $\varepsilon$-pseudospectra with various  $\varepsilon$ for tensor $\mathcal{C}$ are also presented in the same figure from which we can find that these boundaries are reasonably tight in locating those real or complex T-eigenvalues.
\end{example}

\begin{figure}[htb]
	\centering
	\includegraphics[width=\textwidth]{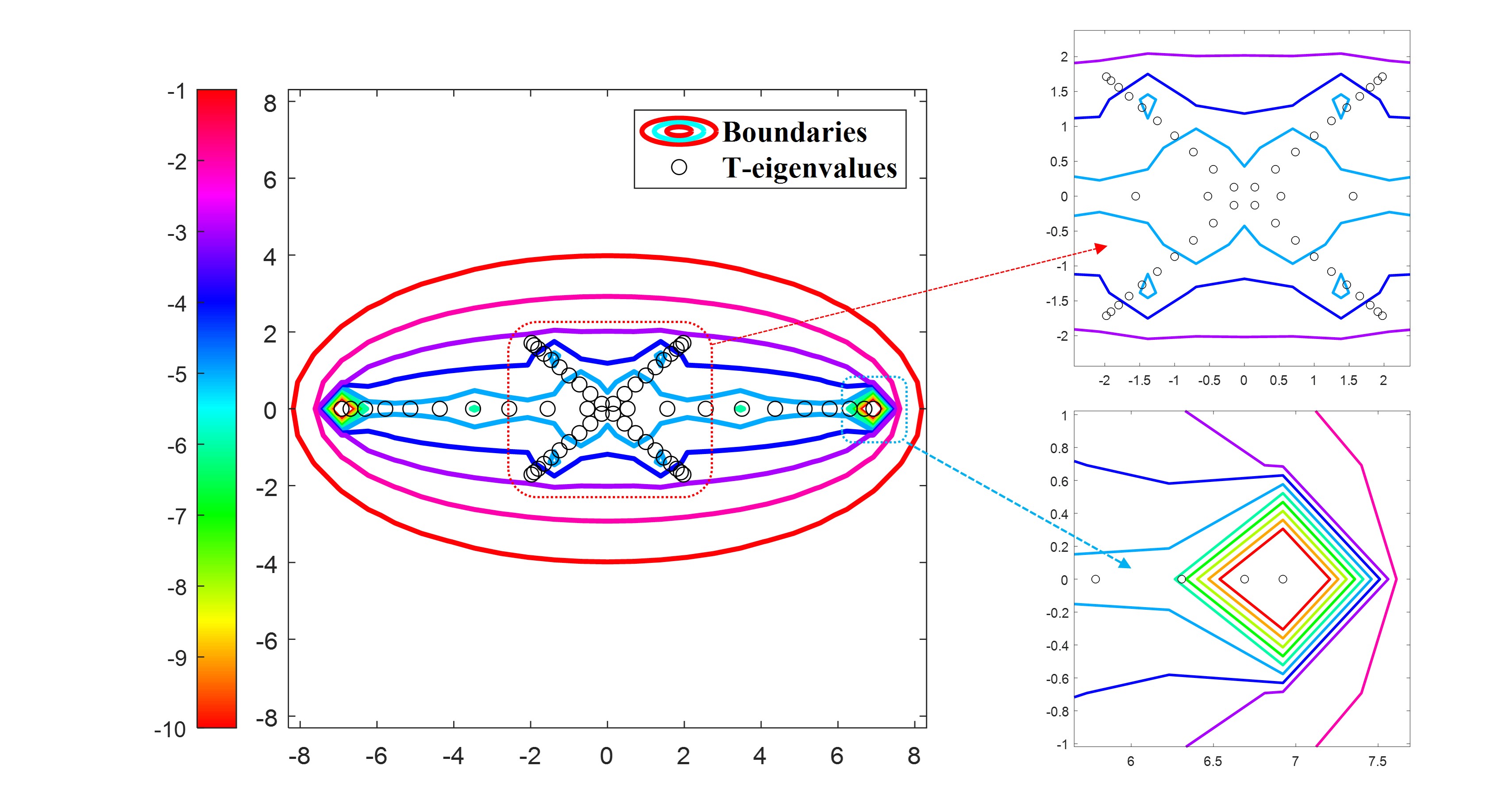}
	\caption{Boundaries of pseudospectra $\Lambda_{\varepsilon}(\mathcal{C})$ for  $\varepsilon=10^{-1}, 10^{-2}, \ldots, 10^{-10}$ under the condition that $C_3 = 2C_2 = 4C_1 = 4T_{pz}$.  The T-eigenvalues of tensor $\mathcal{C}$ are plotted as circles `$\circ$'.} \label{figure_sym000}
\end{figure}

\begin{example}
	 Let $F_1 = T_{pz}, F_2 = 10T_{pz}$ and $F_3 = \operatorname{eye}(N)$. Suppose that they are the three frontal slices of the tensor $\mathcal{F}\in \mathbb{R}^{20\times 20\times 3}$. By a  calculation without too much difficulty, one can get all T-eigenvalues of tensor $\mathcal{F}$. They spread along the complex plane with a shape like six-legged starfish; see Fig. \ref{figure_sym0000} for an intuitive display. To show the results of $\varepsilon$-pseudospectra, by setting similar parameters in plotting those  boundaries, we can also get relatively satisfactory results, just as the results shown in Fig. \ref{figure_sym0000}.
\end{example}
\begin{figure}[h]
	\centering
	\includegraphics[width=\textwidth]{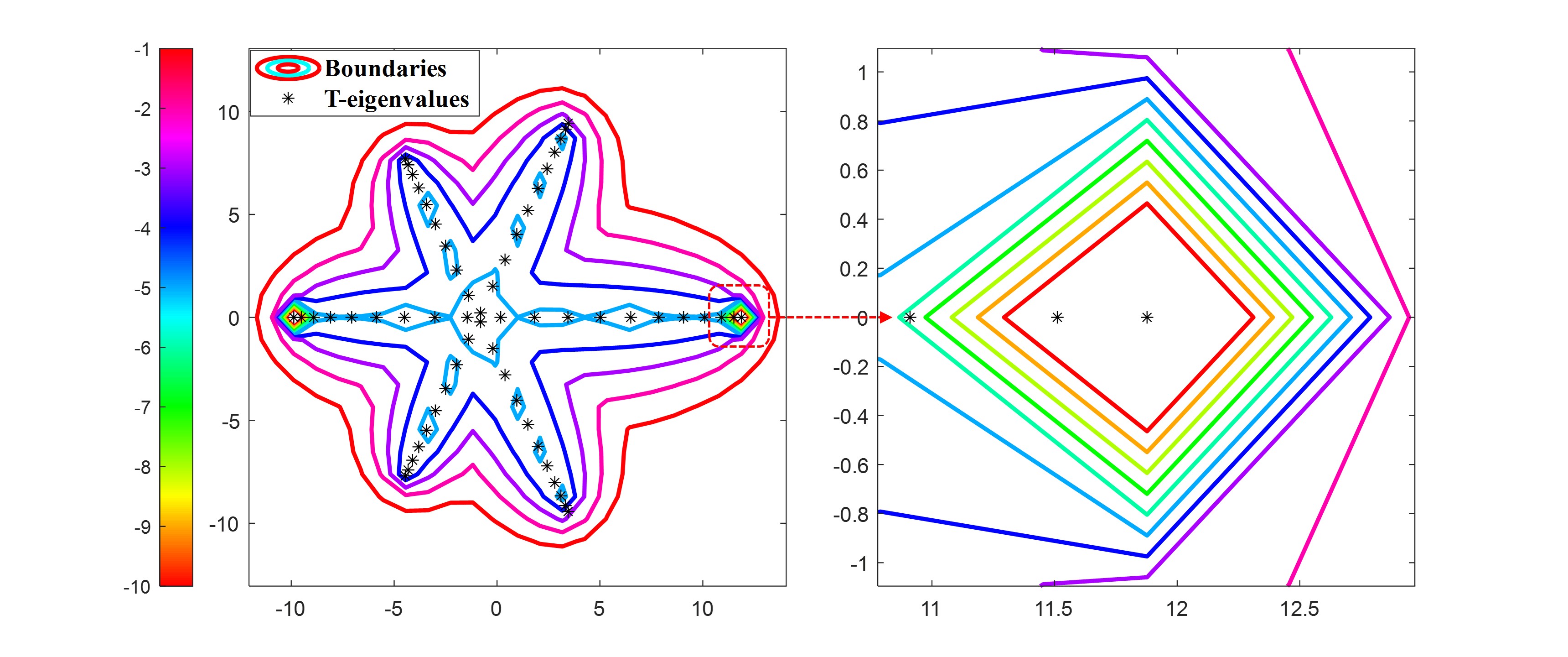}
	\caption{Boundaries of pseudospectra $\Lambda_{\varepsilon}(\mathcal{F})$ for  $\varepsilon=10^{-1}, 10^{-2}, \ldots, 10^{-10}$ under the condition that $F_1 = T_{pz}, F_2= 10 T_{pz}$ and $F_3 = \operatorname{eye}(N)$.  The T-eigenvalues of tensor $\mathcal{F}$ are plotted by `$*$'. } \label{figure_sym0000}
\end{figure}

\begin{example}
	We consider the validation of T-positive definiteness for third-order symmetric tensors 
	$\mathcal{A}$ through the application of the pseudospectra localizations strategy. Let 	$\mathcal{A}$ be a symmetric tensor with three frontal slices 
		$$
	A_1 = \left(
	\begin{array}{cc}
		20 & 2  \\
		2 & 20 
	\end{array}
	\right), A_2 = \left(
		\begin{array}{cc}
	-4 & 1  \\
		2 & -4
	\end{array}
	\right),   A_3 = \left(
	\begin{array}{cc}
	-4 & 2  \\
	1 & -4
\end{array}
	\right).
	$$
With the help of the normalized DFT matrix	apply to $\operatorname{bcirc}(\mathcal{A})$, we derive the block diagonal matrix $A = \operatorname{diag}(A^{(1)},A^{(2)},A^{(3)})$ with 
	$$
A^{(1)} = \left(
\begin{array}{cc}
	12 & 5  \\
	5 & 12 
\end{array}
\right), 
A^{(2)} = \left(
\begin{array}{cc}
	24& \frac{1}{2}+ \frac{\sqrt{3}}{2}i \\
	\frac{1}{2}- \frac{\sqrt{3}}{2}i & 24
\end{array}
\right),   
A^{(3)} = \left(
\begin{array}{cc}
	24& \frac{1}{2}- \frac{\sqrt{3}}{2}i \\
	\frac{1}{2}+ \frac{\sqrt{3}}{2}i & 24
\end{array}
\right).
$$
According to  Theorem \ref{Pseu_Loca}, it is observed that all  T-eigenvalues of the tensor $\mathcal{A}$ belong to the eigenvalue localization set 
$$\Gamma_{\varepsilon}(\mathcal{A})=\{z \in \mathbb{C}:|z-12| \leq 5+\varepsilon\} \bigcup\{z \in \mathbb{C}:|z-24| \leq1+\varepsilon\}$$
	where $\varepsilon\geq 0$. 
	Obviously, taking $\varepsilon = 0$, the $\varepsilon$-pseudospectra Gershgorin sets reduce to the T-spectrum $\Lambda(\mathcal{A})$ of tensor $\mathcal{A}$, while the localization set  $\Gamma_{\varepsilon}(\mathcal{A})$ becomes the Gersgorin sets  $\Gamma(\mathcal{A})$ given in  Theorem \ref{GerThm}. In this case, all sets lie in the open right-half complex plane, and hence the tensor $\mathcal{A}$ is T-positive definite.  
	
	Furthermore, by Theorem \ref{Pseu_Loca}, one can explore more positive definite tensors that surround the positive definite tensor $\mathcal{A}$. We only need to ensure that the localization set  $\Gamma_{\varepsilon}(\mathcal{A})$ lies in the open right-half complex plane, which implies that only the condition $\varepsilon < 7$ should be satisfied.  Therefore, for any tensor $\mathcal{E}$ satisfying $\|\mathcal{E}\| <7$, the tensor $\mathcal{A}+\mathcal{E}$ is T-positive definite if it is symmetric. To sum up, through the application of the $\varepsilon$-pseudospectral theory on third tensors, it becomes feasible to explore additional T-positive definite tensors surrounding a given T-positive definite tensor.
	
\end{example}

\begin{remark}
	Hundreds of years have witnessed the  power of the eigenvalue tool of matrices not only useful in practice but also fundamental in concept \cite{trefethen2005spectra}. However, in the nonnormal matrix case, eigenvalue analysis may reveal little significance and therefore pseudospectra analysis springs up to  remedy such a situation. In the above examples for tensor case,  the tensor $\mathcal{A}$ is not normal, that is, $\mathcal{A}*\mathcal{A}^{\mathrm{H}} \neq \mathcal{A}^{\mathrm{H}} * \mathcal{A}$. By an invertible tensor transformation $\mathcal{D}$ under tensor-tensor multiplication sense, we get a normal tensor $\mathcal{S}$ with all real T-eigenvalues. The first two examples show that the  pseudospectra of $\mathcal{A}$  lie a little far from the real axis that all T-eigenvalues spread on. This coincides  with the results of matrix case greatly \cite{trefethen2005spectra}. But the last two examples indeed indicate that pseudospectra tool could give a satisfactory capability of approximating all T-eigenvalues, especially for tensor $\mathcal{F}$. 
	
\end{remark}

\section{Conclusions}
This paper considers perturbation analysis to T-eigenvalues of third-order tensors, introducing generalizations of the Gershgorin disc, Bauer-Fike, and Kahan theorems and so on. The pseudospectral theory, including four characterizations and properties analysis as well as insightful visualizations of tensor $\varepsilon$-pseudospectra and application to seek more T-positive definite tensors,  on third tensors is presented.  Future work will explore the corresponding multilinear time-invariant systems and rectangular T-eigenvalue problems with dimensions $m\times n\times \ell$, focusing on their theoretical analysis and computational aspects.

\section*{Data availability}
No data was used for the research described in the article.

\section*{Declaration of Competing Interest}

The authors declare that they have no known competing financial interests or personal relationships that could have appeared to influence the work reported in this paper.

\section*{Acknowledgments}
\addcontentsline{toc}{section}{Acknowledgments}
{\small
	The authors would like to thank the handling editor and two referees for their very detailed comments.
	Changxin Mo acknowledges support from the National Natural Science Foundation of China (Grant No. 12201092),  the Natural Science Foundation Project of CQ CSTC (Grant No. CSTB2022NSCQ-MSX0896), the Science and Technology Research Program of Chongqing Municipal Education Commission 
	(Grant No. KJQN202200512),  the Chongqing Talents Project (Grant No. cstc2022ycjh-bgzxm0040), and the Research Foundation of Chongqing Normal University (Grant No. 21XLB040), P. R. of China. \\
\indent 	Weiyang Ding's research is supported by the Science and Technology
Commission of Shanghai Municipality under grants 23ZR1403000, 20JC1419500, and 2018SHZDZX0. \\
\indent	Yimin Wei is supported by the National Natural Science Foundation of China under Grant 12271108, the Ministry of Science and Technology of China under grant G2023132005L and   the Science and Technology Commission of Shanghai Municipality under grant 23JC1400501.
}

%%%%%%%%%%%%%%%%%%%%%%%%%%%%%%%%%%%%%%%%%%%%%%%%%%%%%%%%%%%%%%%%%%%%%%%%%%%%%%%%%%%%%%%%%%%%%%%%%%%%%%%%%%%%%%%%%%%%%%%%%%%%%%%%%%%%%%%%%%%%%%%%%%%%%%%%%%%%%%%%%%%%%%%%%%%%%%%%%%%%%%%%%%%%%%%

\section*{Appendix}
\textbf{Proof of Theorem \ref{Theorem Schur}}
\begin{proof} 
	The theorem is evident when $\mu \in \Lambda(\mathcal{A})$, as the left-hand sides  of \eqref{General BF-theorem} and \eqref{General BF-theorem-p} vanish. 
	Therefore,  we assume that $\mu \notin \Lambda(\mathcal{A})$. 
	By Lemma \ref{Lemmabcirc}, we can see that
	$$
	\mu I_{mn} - \operatorname{bcirc}(\mathcal{A} + \varepsilon \mathcal{B}) = 	\mu I_{mn} - \operatorname{bcirc}(\mathcal{A}) -  \operatorname{bcirc}(\varepsilon\mathcal{B}),
	$$
	and  moreover $\mu I_{mn} - \operatorname{bcirc}(\mathcal{A} + \varepsilon \mathcal{B}) $ is singular. This means that
	\begin{equation}\label{Exp314}
		(F_n^{\mathrm{H}} \otimes I_m)  \operatorname{bcirc}(\mathcal{Q})^{-1}	  [	\mu I_{mn} - \operatorname{bcirc}(\mathcal{A}) -  \operatorname{bcirc}(\varepsilon\mathcal{B})]
		\operatorname{bcirc}(\mathcal{Q}) (F_n \otimes I_m)	
	\end{equation}
	is also singular since the matrices multiplied on the left and right sides are nonsingular.
	Notice that 
	\begin{equation*}
		\begin{aligned}
			& (F_n^{\mathrm{H}} \otimes I_m) 	  \operatorname{bcirc}(\mathcal{Q})^{-1}  \operatorname{bcirc}(\mathcal{A})  \operatorname{bcirc}(\mathcal{Q})  
			(F_n  \otimes I_m)\\
			= &
			(F_n^{\mathrm{H}} \otimes I_m) 	 [
			\operatorname{bcirc}(\mathcal{D}) + \operatorname{bcirc}(\mathcal{N}) ]
			(F_n  \otimes I_m)\\
			= & \left[\begin{array}{llll}
				D^{(1)} & & & \\
				& D^{(2)} & & \\
				& & \ddots & \\
				& & & D^{(n)}
			\end{array}\right]
			+
			\left[\begin{array}{llll}
				N^{(1)} & & & \\
				& N^{(2)} & & \\
				& & \ddots & \\
				& & & N^{(n)}
			\end{array}\right]\\
			:=& D+N.
		\end{aligned}
	\end{equation*}
	Therefore \eqref{Exp314} can be rewritten as
	\begin{equation*}
		\mu I_{mn} - D - N - (F_n^{\mathrm{H}} \otimes I_m) 	  \operatorname{bcirc}(\mathcal{Q})^{-1}  \operatorname{bcirc}(\varepsilon \mathcal{B})  \operatorname{bcirc}(\mathcal{Q})  
		(F_n  \otimes I_m),
	\end{equation*}
	and thus the following matrix
	\begin{equation}\label{Exp313}
		I_{mn}  - (\mu I_{mn} - D - N)^{-1} (F_n^{\mathrm{H}} \otimes I_m)   \operatorname{bcirc}(\mathcal{Q})^{-1}  \operatorname{bcirc}(\varepsilon \mathcal{B})  \operatorname{bcirc}(\mathcal{Q})  
		(F_n \otimes I_m)
	\end{equation}
	is singular. 
	
	By the assumption  that $|N|^q = 0$  and 
	note that  $\mu I_{mn} - D $ is a nonsingular diagonal matrix,  it follows that $((\mu I_{mn} - D)^{-1}N)^q = 0$. Hence,
	\begin{equation*}
		((\mu I_{mn}-D)-N)^{-1}=\sum_{k=0}^{q-1}\left((\mu I_{mn}-D)^{-1} N\right)^{k}(\mu I_{mn}-D)^{-1},
	\end{equation*}
	and 
	\begin{equation*}
		\|((\mu I_{mn}-D)-N)^{-1}\| \leq
		\frac{1}{	\min _{\lambda \in \Lambda(\mathcal{A})}|\lambda-\mu|}
		\sum_{k=0}^{q-1}\left(\frac{\|N\|}{	\min _{\lambda \in \Lambda(\mathcal{A})}|\lambda-\mu|}\right)^{k}
	\end{equation*}
	holds under the $1$-, $2$- and $\infty$-norms cases.
	
	If 	$\min _{\lambda \in \Lambda(\mathcal{A})}|\lambda-\mu|\geq1$, then 
	\begin{equation*}
		\|((\mu I_{mn}-D)-N)^{-1}\| \leq
		\frac{1}{	\min _{\lambda \in \Lambda(\mathcal{A})}|\lambda-\mu|}
		\sum_{k=0}^{q-1}\|N\|^{k}.
	\end{equation*}
	In the case of  $\min _{\lambda \in \Lambda(\mathcal{A})}|\lambda-\mu|<1$, then 
	\begin{equation*}
		\|((\mu I_{mn}-D)-N)^{-1}\| \leq
		\frac{1}{	(\min _{\lambda \in \Lambda(\mathcal{A})}|\lambda-\mu|)^q}
		\sum_{k=0}^{q-1}\|N\|^{k}.
	\end{equation*}
	By (\ref{Exp313}), we can obtain 
	\begin{equation*}
		\begin{aligned}
			1 & \leq\left\|(\mu I_{mn} - D - N)^{-1} (F_n^{\mathrm{H}} \otimes I_m) \cdot 	  \operatorname{bcirc}(\mathcal{Q})^{-1}  \operatorname{bcirc}(\varepsilon \mathcal{B})  \operatorname{bcirc}(\mathcal{Q})  
			\cdot  (F_n \otimes I_m)\right\| \\
			&=\left\|(\mu I_{mn} - D - N)^{-1}\right\|\| \operatorname{bcirc}(\varepsilon \mathcal{B}) \|\|F_n^{\mathrm{H}} \otimes I_m\|\|F_n \otimes I_m\| \|\operatorname{bcirc}(\mathcal{Q})^{-1} \|\|\operatorname{bcirc}(\mathcal{Q}) \|.
		\end{aligned}
	\end{equation*}
	
	Combining the above theoretical analyses, 
	in the spectral norm case, we can see
	$$
	\min _{\lambda \in \Lambda(\mathcal{A})}|\lambda-\mu| \leq \| \operatorname{bcirc}(\varepsilon \mathcal{B}) \|_2 \sum_{k=0}^{q-1}\|N\|_2^{k} 
$$
or
$$
	(\min _{\lambda \in \Lambda(\mathcal{A})}|\lambda-\mu|)^q \leq \| \operatorname{bcirc}(\varepsilon \mathcal{B}) \|_2 \sum_{k=0}^{q-1}\|N\|_2^{k}
	$$
	for 
	$$\min _{\lambda \in \Lambda(\mathcal{A})}|\lambda-\mu|\geq1 
	\quad \text{or } \quad 
	\min _{\lambda \in \Lambda(\mathcal{A})}|\lambda-\mu|<1,
	$$
	respectively. 
	Let $\theta =  \| \operatorname{bcirc}(\varepsilon \mathcal{B}) \|_2 \sum_{k=0}^{q-1}\|N\|_2^{k} $. Then we get the result \eqref{General BF-theorem} for the spectral norm. Similar to the proof given in Theorem \ref{Bauer Ten}, the Frobenius norm case can be obtained easily.
	
	For the $1$- and $\infty$-norms, by using (\ref{Exp313}) again we get 
	$$
	\min _{\lambda \in \Lambda(\mathcal{A})}|\lambda-\mu| \leq \max\{\theta_p, \theta_p^{1/q}\},
	$$
	where
	$$
	\theta_p = \| \operatorname{bcirc}(\varepsilon \mathcal{B}) \|_p \kappa_{p}(\mathcal{Q}) \kappa_{p}(F_n \otimes I_m) \sum_{k=0}^{q-1}\|N\|_2^{k}. 
	$$
	The proof is completed. \qed
\end{proof}

\noindent \textbf{Proof of Theorem \ref{General two}}
\begin{proof}
	We only need to consider the case that $\mu$ is not a T-eigenvalue of $\mathcal{A}$. Hence $\mu I_{mn} - \tilde{D}-\tilde{N}$ is nonsingular.  Similar as (\ref{Exp313}), the matrix
	$$
	I_{mn} - (\mu I_{mn}- \tilde{D} - \tilde{N})^{-1} (F_n \otimes I_m) 	  \operatorname{bcirc}(\mathcal{X})^{-1}  \operatorname{bcirc}(\varepsilon \mathcal{B})  \operatorname{bcirc}(\mathcal{X})  
	(F_n^{\mathrm{H}} \otimes I_m)
	$$
	is singular.  By using a similar proof process given for Theorem \ref{Theorem Schur} in the above, we could get the conclusion. \qed
\end{proof}

\noindent \textbf{Proof of Theorem \ref{ProPseu}}
\begin{proof}
	To prove the assertion of (I), we utilize \eqref{Factorization of bcirc(A)} and Remark \ref{Remark4-1}, which allow us to exploit the properties of $\Lambda_{\varepsilon}(A^{(i)})$ for each matrix $A^{(i)}$. It is well-known that $\Lambda_{\varepsilon}(A^{(i)})$ is nonempty, open, and bounded, with at most $m$ connected components, each containing one or more eigenvalues of $A^{(i)}$ \cite[Theorem 2.4]{trefethen2005spectra}. Consequently, these same properties also hold for the given tensor $\mathcal{A}$ by the above analysis. Additionally, the number of connected components is bounded by $nm$ due to the relationship expressed by
	$$
	\Lambda_{\varepsilon}(\mathcal{A})= \bigcup_{i=1}^{n}	\Lambda_{\varepsilon}(A^{(i)}).
	$$

	Now, we proceed to part (II).  Denote  $\left(F_{n}^{\mathrm{H}} \otimes I_{m}\right)   \operatorname{bcirc}(\mathcal{A}) \left(F_{n}\otimes I_{m}\right) :=A$. 
	First, note that 
	\begin{align*}
		\operatorname{bcirc}(\mathcal{A}+c\mathcal{E}) 
		&= 
		\left[\begin{array}{ccccc}
			{A_{1} + cI_m} & {A_{n}} & {A_{n-1}} & {\cdots} & {A_{2}} \\
			{A_{2}} & {A_{1} + c I_m } & {A_{n}} & {\cdots} & {A_{3}} \\
			{\vdots} & {\ddots} & {\ddots} & {\ddots} & {\vdots} \\
			{A_{n}} & {A_{n-1}} & {\ddots} & {A_{2}} & {A_{1} + cI_m } 
		\end{array}\right]\\
		& =  \operatorname{bcirc}(\mathcal{A}) + c \operatorname{bcirc}(\mathcal{E}) \\
		& =  \left(F_{n} \otimes I_{m}\right)   A  \left(F_{m}^{\mathrm{H}}\otimes I_{n}\right) + c[\left(F_{m} \otimes I_{n}\right)   I_{mn}  \left(F_{m}^{\mathrm{H}}\otimes I_{n}\right)]\\
		& =  \left(F_{n} \otimes I_{m}\right)   (A+cI_{mn})  \left(F_{n}^{\mathrm{H}}\otimes I_{m}\right)\\
		&= \left(F_{n} \otimes I_{m}\right) 
		\left[\begin{array}{cccc}
			{A^{(1)}+cI_m} & {} & {} & {} \\
			{} & {A^{(2)}+cI_m} & {} & {} \\
			{} & {} & {\ddots} & {} \\
			{} & {} & {} & {A^{(n)}+cI_m}
		\end{array}\right]
		\left(F_{n}^{\mathrm{H}}\otimes I_{m}\right).
	\end{align*}
	Therefore, 	for any $c\in \mathbb{C}$, we have
	$$
	\Lambda_{\varepsilon}(\mathcal{A}+c)= \bigcup_{i=1}^{n}	\Lambda_{\varepsilon}(A^{(i)}+cI_m) = \bigcup_{i=1}^{n}	[\Lambda_{\varepsilon}(A^{(i)}) +c] = c + \Lambda_{\varepsilon}(\mathcal{A}).
	$$ 	
	We complete the proof of this part.
	
	For part (III), by Lemma \ref{Lemmabcirc}, we know that 
	$$
	\operatorname{bcirc}(c \mathcal{A})=c \operatorname{bcirc}(\mathcal{A}) = c  [(F_{n} \otimes I_{m})   A  \left(F_{n}^{\mathrm{H}}\otimes I_{m}\right)]   =  \left(F_{n} \otimes I_{m}\right)   (cA)  \left(F_{n}^{\mathrm{H}}\otimes I_{m}\right),
	$$
	which implies that 
	$$
	\Lambda_{|c| \varepsilon}(c \mathcal{A})  =  \bigcup_{i=1}^{n}	\Lambda_{|c|\varepsilon}(cA_i) = \bigcup_{i=1}^{n}c	\Lambda_{\varepsilon}(A_i) = c \bigcup_{i=1}^{n}	\Lambda_{\varepsilon}(A_i),
	$$
	since for any nonzero  $c\in \mathbb{C}$ and matrix $A\in \mathbb{C}^{m\times m}$, the following equality $$\Lambda_{|c|\varepsilon}(cA) = c	\Lambda_{\varepsilon}(A)$$
	holds \cite[Theorem 2.4]{trefethen2005spectra}. Thus we get the result that $\Lambda_{|c| \varepsilon}(c \mathcal{A})=c \Lambda_{\varepsilon}(\mathcal{A})$  for any nonzero $c \in \mathbb{C}$.
	
	Now, we prove the last part of this theorem. By Lemma \ref{Lemmabcirc}, we know that 
	$$
	\operatorname{bcirc}\left(\mathcal{A}^{\mathrm{H}}\right) = \left(F_{n} \otimes I_{m}\right)   A^{\mathrm{H}}  \left(F_{n}^{\mathrm{H}}\otimes I_{m}\right). 
	$$
	Therefore, 
	$$
	\Lambda_{\varepsilon}\left(\mathcal{A}^{{\mathrm{H}} }\right) = \bigcup_{i=1}^{n}	\Lambda_{\varepsilon}((A^{(i)})^{\mathrm{H}} ) = \bigcup_{i=1}^{n}\overline{	\Lambda_{\varepsilon}(A^{(i)})}
	=
	\overline{	\bigcup_{i=1}^{n}\Lambda_{\varepsilon}(A^{(i)})} = \overline{  \Lambda_{\varepsilon}\left(\mathcal{A}\right) },
	$$
	where the conclusion $\Lambda_{\varepsilon}(A^{\mathrm{H}})  = \overline{	\Lambda_{\varepsilon}(A)}$ under the two-norm case for any matrix $A\in \mathbb{C}^{m\times m}$ \cite[Theorem 2.4]{trefethen2005spectra}  is applied in the second equality. \qed
\end{proof}

\noindent \textbf{Proof of Theorem \ref{PseuNormal}}

\begin{proof}
	If $\lambda$ is a T-eigenvalue of tensor $\mathcal{A}$, then it is an eigenvalue of the matrix $\operatorname{bcirc}(\mathcal{A})$. Therefore, for any $\mu \in \mathbb{C}$, $\lambda + \mu$ is an eigenvalue of $\operatorname{bcirc}(\mathcal{A}) + \mu I$.  Note that $\|\mu I\| = |\mu|$,  and by the definition of pseudospectra on tensors, we obtain
	$\lambda + \mu \in \Lambda_{\varepsilon}(\mathcal{A})$ for any $|\mu| < \varepsilon$. Thus, we have completed the proof of (\ref{notnormal-subset}).
	
	For the normal tensor case, by  Lemmas \ref{normal diagonal}  and \ref{Lemmabcirc}, we obtain 
	$$
	\operatorname{bcirc}(\mathcal{U})^{\mathrm{H}} \operatorname{bcirc}(\mathcal{A})(\operatorname{bcirc}(\mathcal{U})) 
	= \operatorname{bcirc}(\mathcal{D}), 
	$$
	and 
	\begin{equation}\label{diagD}
		(F_n^{\mathrm{H}}  \otimes I_m) \operatorname{bcirc}(\mathcal{D})  (F_n\otimes I_m) = 
		\left[\begin{array}{cccc}
			{D^{(1)}} & {} & {} & {} \\
			{} & {D^{(2)}} & {} & {} \\
			{} & {} & {\ddots} & {} \\
			{} & {} & {} & {D^{(n)}}
		\end{array}\right] : = D,
	\end{equation}
	in which $D^{(i)}$ is diagonal for $i = 1, \cdots, n$ by Lemma \ref{lemmaFdiagonal iff}.  Also note that $\|\cdot\|=\|\cdot\|_{2}$, we may assume directly that $\mathcal{A}$ is F-diagonal. Therefore, the diagonal entries of $\operatorname{bcirc}(\mathcal{A})$ are exactly the T-eigenvalues. As we all know,  the $\varepsilon$-pseudospectra is just the union of the open $\varepsilon$-balls about the points of the spectra for any normal matrix; equivalently, we have
	\begin{equation*}
		\left\|(z-\operatorname{bcirc}(\mathcal{A}))^{-1}\right\|_{2}=\frac{1}{\operatorname{dist}(z, \Lambda(\operatorname{bcirc}(\mathcal{A})))},
	\end{equation*}
	which implies 
	$$
	\operatorname{dist}(z, \Lambda(\operatorname{bcirc}(\mathcal{A}))) < \varepsilon
	$$
	by the $\varepsilon$-pseudospectra of tensors. We get the conclusion since $\Lambda(\mathcal{A})+\Delta_{\varepsilon}$ is the same as $\{z: \operatorname{dist}(z, \Lambda(\mathcal{A}))<\varepsilon\}$. \qed
\end{proof}

%%%%%%%%%%%%%%%%%%%%%%%%%%%%%%%%%%%%%%%%%%%%%%%%%%%%%%%%%%%%%%%%%%%%%%%%%%%%%%%%%%%%%%%%%%%%%%%%%%%%%%%%%%%%%%%%%%%%%%%%%%%%%%%%%%%%%%%%%%%%%%%%%%%%%%%%%%%%%%%%%%%%%%%%%%%%%%%%%%%%%%%%

%%%%%%%%%%%%%%%%%%%%%%%%%%%%%%%%%%%%%%%%%%%%%%%%%%%%%%%%%%%%%%%%%%%%%%%%%%%%%%%%%%%%%%%%%%%%%%%%%%%%%

\begin{thebibliography}{}
	
\bibitem{Bauer-Fike}
Bauer, F.L., Fike, C.T.: Norms and exclusion theorems.
\newblock Numer. Math. \textbf{2}, 137--141 (1960)


\bibitem{saad2023}
Beik, F., Saad, Y.: On the tubular eigenvalues of third-order tensors.
\newblock arXiv preprint  
\newblock  \href{https://arxiv.org/abs/2305.06323}{arXiv:2305.06323}
\newblock  (2023)

\bibitem{braman2010thirdorder}
Braman, K.: Third-order tensors as linear operators on a space of matrices.
\newblock Linear Algebra Appl. \textbf{433}(7), 1241--1253 (2010)


\bibitem{brazell_solving_2013}
Brazell, M., Li, N., Navasca, C., Tamon, C.: Solving multilinear systems via
tensor inversion.
\newblock SIAM J. Matrix Anal. Appl. \textbf{34}(2), 542--570 (2013)


\bibitem{cao2022perturbation}
Cao, Z., Xie, P.: Perturbation analysis for t-product-based tensor inverse,
{M}oore-{P}enrose inverse and tensor system.
\newblock Commun. Appl. Math. Comput. \textbf{4}(4), 1441--1456 (2022)

\bibitem{cao2021tensor}
Cao, Z., Xie, P.: On some tensor inequalities based on the t-product.
\newblock Linear Multilinear Algebra \textbf{71}(3), 377--390 (2023)


\bibitem{chang2022t2}
Chang, S.Y., Wei, Y.: T-product tensors---part {II}: tail bounds for sums of
random {T}-product tensors.
\newblock Comput. Appl. Math. \textbf{41}(3), Paper No. 99, 32 (2022)


\bibitem{chang2022t}
Chang, S.Y., Wei, Y.: T-square tensors---{P}art {I}: inequalities.
\newblock Comput. Appl. Math. \textbf{41}(1), Paper No. 62, 27 (2022)

\bibitem{Chen2021}
Chen, C., Surana, A., Bloch, A. M., Rajapakse, I.:
Multilinear control systems theory.
\newblock SIAM J. Control Optim. \textbf{59}(1), 749--776 (2021)

\bibitem{chen2023perturbations}
Chen, J., Ma, W., Miao, Y., Wei, Y.: Perturbations of {T}ensor-{S}chur
decomposition and its applications to multilinear control systems and facial
recognitions.
\newblock Neurocomputing \textbf{547} 
\newblock Art. 126359, (2023)

\bibitem{1986Generalization}
Chu, K.-W.E.: Generalization of the {B}auer-{F}ike theorem.
\newblock Numer. Math. \textbf{49}(6), 685--691 (1986)


\bibitem{cui2021perturbation}
Cui, Y.-N., Ma, H.-F.: The perturbation bound for the {T}-{D}razin inverse of
tensor and its application.
\newblock Filomat \textbf{35}(5), 1565--1587 (2021)

\bibitem{Davis1979circulant}
Davis, P. J.: Circulant Matrices, 2nd edn.
\newblock Wiley, New York (1979)





\bibitem{Golub2013matrix}
Golub, G.H., Van~Loan, C.F.: Matrix Computations, 4th edn.
\newblock  Johns Hopkins
University Press, Baltimore, MD (2013)

\bibitem{greenbaum2020firstorder}
Greenbaum, A., Li, R.C., Overton, M.L.: First-order perturbation theory for
eigenvalues and eigenvectors.
\newblock SIAM Rev. \textbf{62}(2), 463--482 (2020)

\bibitem{el2023spectral}
Hachimi, A.E., Jbilou, K., Ratnani, A., Reichel, L.: Spectral computation with
third-order tensors using the t-product.
\newblock Appl. Numer. Math. \textbf{193}, 1--21 (2023)


\bibitem{han2023t}
Han, F., Miao, Y., Sun, Z., Wei, Y.: {T-ADAF}: Adaptive data augmentation
framework for image classification network based on tensor {T}-product
operator.
\newblock Neural Process.  Lett. \textbf{55}, 10993--11016 (2023)


\bibitem{Kilmer2013SIAM}
Hao, N., Kilmer, M.E., Braman, K., Hoover, R.C.: Facial recognition using
tensor-tensor decompositions.
\newblock SIAM J. Imaging Sci. \textbf{6}(1), 437--463 (2013)


\bibitem{Horn2013}
Horn, R.A., Johnson, C.R.: Matrix Analysis, 2nd edn.
\newblock Cambridge University Press, Cambridge (2013)

\bibitem{kato2013}
Kato, T.: Perturbation Theory for Linear Operators, 
\newblock Springer-Verlag New York, Inc., New York (1966)

\bibitem{Kilmer2011report}
Kilmer, M.E., Braman, K., Hao, N.: Third-order tensors as operators on
matrices: A theoretical and computational framework with applications in
imaging.
\newblock Technical Report 2011-01, Tufts University (2011).
\newblock  \href{https://www.cs.tufts.edu/t/tr/techreps/TR-2011-01}{https://www.cs.tufts.edu/t/tr/techreps/TR-2011-01}



\bibitem{kilmer2013third}
Kilmer, M.E., Braman, K., Hao, N., Hoover, R.C.: Third-order tensors as
operators on matrices: a theoretical and computational framework with
applications in imaging.
\newblock SIAM J. Matrix Anal. Appl. \textbf{34}(1), 148--172 (2013)


\bibitem{kilmertensortensor}
Kilmer, M.E., Horesh, L., Avron, H., Newman, E.: Tensor-tensor algebra for
optimal representation and compression of multiway data.
\newblock Proc. Natl. Acad. Sci. USA \textbf{118}(28), Paper No. e2015851,118,
12 (2021)


\bibitem{Kilmer2011}
Kilmer, M.E., Martin, C.D.: Factorization strategies for third-order tensors.
\newblock Linear Algebra Appl. \textbf{435}(3), 641--658 (2011)


\bibitem{Kilmer2008third}
Kilmer, M.E., Martin, C.D., Perrone, L.: A third-order generalization of the
matrix {SVD} as a product of third-order tensors.
\newblock Technical Report 2008-4, Tufts University (2008).
\newblock  \href{https://www.cs.tufts.edu/t/tr/techreps/TR-2008-4}{https://www.cs.tufts.edu/t/tr/techreps/TR-2008-4}


\bibitem{KostiPseudospectra2016}
Kosti\'{c}, V. R., Cvetkovi\'{c}, Lj., Cvetkovi\'{c}, D. Lj.:
Pseudospectra localizations and their applications.
\newblock Numer. Linear Algebra Appl. \textbf{23}(2), 356--372 (2016)



\bibitem{LiLiuWei2019}
Li, C., Liu, Q., Wei, Y.: Pseudospectra localizations for generalized tensor
eigenvalues to seek more positive definite tensors.
\newblock Comput. Appl. Math. \textbf{38}(4), Paper No. 183, 22 (2019)


\bibitem{Jin2020}
Liu, W.-H., Jin, X.-Q.: A study on {T}-eigenvalues of third-order tensors.
\newblock Linear Algebra Appl. \textbf{612}, 357--374 (2021)


\bibitem{liu2018improved}
Liu, Y., Chen, L., Zhu, C.: Improved robust tensor principal component analysis
via low-rank core matrix.
\newblock IEEE J. Sel. Top. Signal Process. \textbf{12}(6),
1378--1389 (2018)

\bibitem{liu2022weighted}
Liu, Y., Ma, H.: Weighted generalized tensor functions based on the
tensor-product and their applications.
\newblock Filomat \textbf{36}(18), 6403--6426 (2022)

\bibitem{lu2019tensor}
Lu, C., Feng, J., Chen, Y., Liu, W., Lin, Z., Yan, S.: Tensor robust principal
component analysis with a new tensor nuclear norm.
\newblock IEEE Trans. Pattern Anal. Mach. Intell.
\textbf{42}(4), 925--938 (2019)

\bibitem{Lund2020}
Lund, K.: The tensor t-function: a definition for functions of third-order
tensors.
\newblock Numer. Linear Algebra Appl. \textbf{27}(3), e2288, 17 (2020)


\bibitem{lund2023frechet}
Lund, K., Schweitzer, M.: The {F}r\'{e}chet derivative of the tensor
t-function.
\newblock Calcolo \textbf{60}(3), Paper No. 35, 34 (2023)


\bibitem{Ng2022}
Luo, Y.S., Zhao, X.L., Jiang, T.X., Chang, Y., Ng, M.K., Li, C.:
Self-supervised nonlinear transform-based tensor nuclear norm for
multi-dimensional image recovery.
\newblock IEEE Trans. Image Process. \textbf{31}, 3793--3808 (2022)


\bibitem{Miao2020generalized}
Miao, Y., Qi, L., Wei, Y.: Generalized tensor function via the tensor singular
value decomposition based on the {T}-product.
\newblock Linear Algebra Appl. \textbf{590}, 258--303 (2020)


\bibitem{Miao2020T}
Miao, Y., Qi, L., Wei, Y.: T-{J}ordan canonical form and {T}-{D}razin inverse
based on the {T}-product.
\newblock Commun. Appl. Math. Comput. \textbf{3}(2), 201--220 (2021)


\bibitem{miao2022stochastic}
Miao, Y., Wang, T., Wei, Y.: Stochastic conditioning of tensor functions based
on the tensor-tensor product.
\newblock Pac. J. Optim. \textbf{19}(2), 205--235 (2023)

\bibitem{mo2019z}
Mo, C., Li, C., Wang, X., Wei, Y.: {$Z$}-eigenvalues based structured tensors:
{$\mathcal {M}_z$}-tensors and strong {$\mathcal{M}_z$}-tensors.
\newblock Comput. Appl. Math. \textbf{38}(4), Paper No. 175, 25 (2019)


\bibitem{mo2020time}
Mo, C., Wang, X., Wei, Y.: Time-varying generalized tensor eigenanalysis via
{Z}hang neural networks.
\newblock Neurocomputing \textbf{407}, 465--479 (2020)

\bibitem{Kilmer2020ima}
Newman, E., Kilmer, M.E.: Nonnegative tensor patch dictionary approaches for
image compression and deblurring applications.
\newblock SIAM J. Imaging Sci. \textbf{13}(3), 1084--1112 (2020)


\bibitem{Olson2014circulant}
Olson, B. J., Shaw, S. W., Shi, C., Pierre, C.,  Parker, R. G.:
Circulant matrices and their application to vibration analysis.
\newblock Appl. Mech. Rev. \textbf{66}(4), 040803  (2014)


\bibitem{pakmanesh2022m}
Pakmanesh, M., Afshin, H.: {$M$}-numerical ranges of odd-order tensors based on
operators.
\newblock Ann. Funct. Anal. \textbf{13}(3), Paper No. 37, 22 (2022)


\bibitem{qi2005eigenvalues}
Qi, L.: Eigenvalues of a real supersymmetric tensor.
\newblock J. Symbolic Comput. \textbf{40}(6), 1302--1324 (2005)


\bibitem{qi2021t}
Qi, L., Zhang, X.: T-quadratic forms and spectral analysis of {T}-symmetric
tensors.
\newblock arXiv preprint  
\newblock  \href{https://arxiv.org/abs/2101.10820}{arXiv:2101.10820}
\newblock  (2021)

\bibitem{Rayleigh1927}
Rayleigh, L.: The Theory of Sound. Volume I.
\newblock Macmillan, London (1927)

\bibitem{Rellich1969}
Rellich, F.: Perturbation Theory of Eigenvalue Problems.
\newblock Gordon and Breach Science Publishers, New York-London-Paris (1969)


\bibitem{Schrodinger1926}
Schr\"odinger, E.: {Quantisierung als Eigenwertproblem}.
\newblock Annalen Phys. \textbf{386}(18), 109--139 (1926)


\bibitem{shi2012sharp}
Shi, X., Wei, Y.: A sharp version of {B}auer-{F}ike's theorem.
\newblock J. Comput. Appl. Math. \textbf{236}(13), 3218--3227 (2012)


\bibitem{stewartsun1990matrix}
Stewart, G.W., Sun, J.G.: Matrix Perturbation Theory.
\newblock Computer Science and Scientific Computing. Academic Press, Inc.,
Boston, MA (1990)

\bibitem{Sun1987}
Sun, J.: Matrix Perturbation Analysis (In Chinese).
\newblock Academic Press, Beijing (1987)

\bibitem{tang2023sketch}
Tang, L., Yu, Y., Zhang, Y., Li, H.: Sketch-and-project methods for tensor
linear systems.
\newblock Numer. Linear Algebra Appl. \textbf{30}(2), Paper No. e2470, 32
(2023)

\bibitem{trefethen2005spectra}
Trefethen, L.N., Embree, M.: Spectra and Pseudospectra: The Behavior of Nonnormal Matrices and Operators.
\newblock Princeton University Press, Princeton, NJ (2005)

\bibitem{Turatti2022}
Turatti, E.: On tensors that are determined by their singular tuples.
\newblock SIAM J. Appl. Algebra Geom. \textbf{6}(2), 319--338 (2022)


\bibitem{wang2023solving}
Wang, X., Che, M., Mo, C., Wei, Y.: Solving the system of nonsingular tensor
equations via randomized {K}aczmarz-like method.
\newblock J. Comput. Appl. Math. \textbf{421}, Paper No. 114,856, 15 (2023)


\bibitem{wang2023fixed}
Wang, X., Wei, P., Wei, Y.: A fixed point iterative method for third-order
tensor linear complementarity problems.
\newblock J. Optim. Theory Appl. \textbf{197}(1), 334--357 (2023)


\bibitem{wang2022hot}
Wang, Y., Yang, Y.: Hot-{SVD}: higher order t-singular value decomposition for
tensors based on tensor-tensor product.
\newblock Comput. Appl. Math. \textbf{41}(8), Paper No. 394, 33 (2022)


\bibitem{wei2023neural}
Wei, P., Wang, X., Wei, Y.: Neural network models for time-varying tensor
complementarity problems.
\newblock Neurocomputing \textbf{523}, 18--32 (2023)

\bibitem{wu2020graph}
Wu, T.: Graph regularized low-rank representation for submodule clustering.
\newblock Pattern Recognit. \textbf{100}, Art. 107145, (2020)


\bibitem{yang2023perron}
Yang, Y., Zhang, J.: Perron-{F}robenius type theorem for nonnegative tubal
matrices in the sense of {$t$}-product.
\newblock J. Math. Anal. Appl. \textbf{528}(2), Paper No. 127, 541, 17 (2023)


\bibitem{ZHAO2020137}
Zhao, X.L., Xu, W.H., Jiang, T.X., Wang, Y., Ng, M.K.: Deep plug-and-play prior
for low-rank tensor completion.
\newblock Neurocomputing \textbf{400}, 137--149 (2020)


\bibitem{zheng2020t}
Zheng, M.M., Huang, Z.H., Wang, Y.: T-positive semidefiniteness of third-order
symmetric tensors and {T}-semidefinite programming.
\newblock Comput. Optim. Appl. \textbf{78}(1), 239--272 (2021)

	
	
	%%%%%%%%%%%%%%%%%%%%%%%%%%%%%%%%%%%%%%%%%%%%%%%%%%%%%%%%%%%%%%%%%%%%%%%%%%%%%%%%%%%%%%%%%%%%%%%%%%%%%%%%%%%%%%%%%%%%%%%%%%%%%%%%%%%%%%%%%%%%%%%%%%%%%%%%%%%%%%%%%%%%%%%%%%%%%%%%%%%%%%%%

\end{thebibliography}
\end{document}